\documentclass[11pt]{amsart}

\usepackage{amssymb,amsmath}
\usepackage[]{graphicx}
\usepackage{amssymb, epsfig}
\usepackage{color}
\numberwithin{figure}{section} \numberwithin{equation}{section}

\newtheorem{theorem}{Theorem}[section]
\newtheorem{lemma}[theorem]{Lemma}

\newtheorem{definition}[theorem]{Definition}

\newtheorem{prop}[theorem]{Proposition}
\newtheorem{remark}[theorem]{Remark}
\usepackage{epsf,pstricks,multicol,pst-grad,color,pst-node,pst-text,pst-3d,pst-tree}
\begin{document}
\title[Malliavin Calculus for the Stochastic Cahn-Hilliard/Allen-Cahn]{Malliavin Calculus for the
\\ Stochastic Cahn-Hilliard/Allen-Cahn
Equation\\ with Unbounded Noise Diffusion}
\author{Dimitra C. Antonopoulou}
\address{Dimitra Antonopoulou, Department of Mathematics, University of
Chester, Thornton Science Park, CH2 4NU, Chester, UK, and,
Institute of Applied and Computational Mathematics, FORTH,
GR--711 10 Heraklion, Greece.}
\email{d.antonopoulou@chester.ac.uk}
\author{Dimitris Farazakis}
\address{Dimitris Farazakis, Department of Mathematics and Applied Mathematics, University of Crete,
GR--714 09 Heraklion, Greece.} \email{D.Farazakis@outlook.com}

\author{Georgia Karali}
\address{Georgia Karali, Department of Mathematics and Applied Mathematics, University of Crete,
GR--714 09 Heraklion, Greece, and, Institute of Applied and
Computational Mathematics, FORTH, GR--711 10 Heraklion, Greece.}
\email{gkarali@tem.uoc.gr}
\begin{abstract}
The stochastic partial differential equation analyzed in this
work, is motivated by a simplified mesoscopic physical model for
phase separation. It describes pattern formation due to adsorption
and desorption mechanisms involved in surface processes, in the
presence of a stochastic driving force. This equation is a
combination of Cahn-Hilliard and Allen-Cahn type operators with a
multiplicative, white, space-time noise of unbounded diffusion.
We apply Malliavin calculus, in order to investigate the
existence of a density for the stochastic solution $u$. In
dimension one, according to the regularity result in \cite{AKM},
$u$ admits continuous paths a.s. Using this property, and inspired by a
method proposed in \cite{CW1}, we construct a modified
approximating sequence for $u$, which properly treats the new second
order Allen-Cahn operator. Under a localization argument, we
prove that the Malliavin derivative of $u$ exists locally, and
that the law of $u$ is absolutely continuous, establishing thus
that a density exists.
\end{abstract}
\subjclass{35K55, 35K40, 60H30, 60H15.} \maketitle
\small{\textbf{Keywords:} stochastic partial differential
equations, reaction-diffusion equations, phase transitions,
Malliavin calculus.}

\section{Introduction}
\subsection{The Stochastic Model}


We consider the following stochastic partial differential equation which is given as a combination of
Cahn-Hilliard and Allen-Cahn type equations, perturbed by a
multiplicative space-time noise $\dot{W}$ with a non-linear
diffusion coefficient $\sigma$
\begin{equation} \label{sm}
u_t=-\varrho\Delta\Big{(}\Delta u-f(u)\Big{)} +\Big{(}\Delta
u-f(u)\Big{)}+\sigma(u) \dot{W},\;\;\;\;t>0,\;x\in\mathcal{D},
\end{equation}
where $\mathcal{D}\subset\mathbb{R}^d$, for $d=1,2,3$, is a
bounded spatial domain. Here, $f(u)=u^3-u$ is the derivative of a
double equal-well potential. The constant $\varrho>0$ is a
positive bifurcation parameter referring to an attractive
potential for the related physical model, while the noise
$\dot{W}=\dot{W}(x,t)$ is a space-time white noise in the sense
of Walsh, \cite{W}, given as the formal derivative of a Wiener
process. More specifically, $dW:=W(dx,ds)$ is a $d$-dimensional
space-time white noise, induced by the one-dimensional
$(d+1)$-parameter Wiener process $W$ defined as $W:=\big\{
W(t,x):\; t\in[0,T],\; x\in{\mathcal{D}}\big\}$. The noise
diffusion  $\sigma(u)$ has a sub-linear
 growth of the form
\begin{equation*}
\left. |\sigma(u) \right|\leq C(1+|u|^q),
\end{equation*}
for some $C >0$ and any $q \in (0,\frac{1}{3})$.

The initial and boundary value problem for this equation,
satisfies the initial condition
$$u(x,0)=u_0(x) \mbox{ }     \mbox {in  }   \mathcal{D}, $$ and the next homogeneous Neumann boundary
 conditions
 \begin{equation}\label{neu}
\frac{\partial{u}}{\partial{\nu}}=\frac{\partial{\Delta
u}}{\partial{\nu}}=0 \mbox{   } \mbox {on} \mbox{  }
\partial{\mathcal{D}}\times{[0,T)}.
\end{equation}

The Cahn-Hilliard equation was initially proposed as a simple
model for the description of the phase separation of a binary
alloy, being in a non-equilibrium state, \cite{CH}. Cook in
\cite{C}, extended the deterministic partial differential
equation to a stochastic one by introducing thermal fluctuations
in the form of an additive noise. There exist some interesting
results in the relevant literature on
 existence and uniqueness of solution for the
stochastic problem, as for example in \cite{CW1,PD}, where the
i.b.v.p. was posed on cubic domains, and rectangles, or on
Lipschitz domains of more general topography, \cite{AK}. In
\cite{BMW,PD,CW1,CW2}, the authors considered the version of an
odd polynomial nonlinearity for the potential. Moreover, in
\cite{ABK}, the one-dimensional stochastic Cahn-Hilliard equation
has been approximated by a manifold of solutions and the dynamics
of the stochastic motion of the fronts were described. In
\cite{BMW}, the effect of noise on evolving interfaces during the
initial stage of phase separation was analyzed, while in
\cite{AKO}, the singular limit of the generalized Cahn-Hilliard
equation has been rigorously derived by means of the Hilbert
expansion method, imitating the behavior of a stochastic model.
The sharp interface limit of the Cahn-Hilliard equation with
additive noise has been examined in \cite{ABK2}; in this case,
depending on the noise strength, the chemical potential satisfies
on the limit a deterministic or a stochastic Hele-Shaw problem of
Stefan type. Funaki studied the interface motion and applied a
singular perturbation analysis for the Allen-Cahn equation with
mild noise, when the initial data are close to an instanton,
\cite{Fun95,Fun99}. In the presence of a non-local integral term
the Allen-Cahn equation exhibits the mass conservation property;
 for the dynamics of the mass conserving stochastic Allen-Cahn
equation, we refer to the results presented in \cite{ABBK}.

In the deterministic setting, Karali and Katsoulakis, in
\cite{KK}, introduced a simplified mean field type model written
as a combination of Cahn-Hilliard and Allen-Cahn type equations,
in order to study the effect of diffusion and
adsorption/desorption in the context of surface processes.
Antonopoulou, Karali and Millet in \cite{AKM}, by inserting a
noise term additive in the equation and stemming from the free
energy and thermal fluctuations, derived the stochastic
non-linear equation version of the aforementioned model. There
in, the authors described the physical motivation of such a
stochastic forcing. In addition, they investigated the existence
and regularity of solution for the stochastic
Cahn-Hilliard/Allen-Cahn equation with unbounded noise diffusion,
when posed in dimensions $d=1,2,3$.

Our aim in this work, is to study the existence of a density for
the stochastic solution. The dimensions of the problem in spatial
coordinates are expected to play a crucial role. Note that in
dimensions $d=1$ the stochastic solution has continuous paths
a.s., while in higher dimensions existence of maximal solutions
has been established, \cite{AKM}.

\subsection{The Malliavin derivative}
Let $(\Omega,\mathcal{F},P)$ be a probability space, where
$\Omega$ is a sample space, $\mathcal{F}$ is a $\sigma$-algebra
consisting of subsets of $\Omega$ and $P$ a probability measure
$P:\mathcal{F}\rightarrow[0,1]$, and consider a random variable
$F:\Omega\rightarrow \mathbb{R}$. The Malliavin derivative
measures the rate of change of $F$ as a function of
$\omega\in\Omega$ and implements the idea of differentiating $F$
with respect to the chance parameter $\omega$, \cite{N}. When
$\Omega$ has a topological structure, the derivative operator is
induced by a directional Fr\'echet derivative of $F$ along a
certain direction $\omega_0$ in $\Omega$, of the form
$$\frac{d}{d\varepsilon}F(\omega+\varepsilon
\omega_0)|_{\varepsilon=0},$$ \cite{N}. The function $F$ can be a
stochastic process as for example the solution of a stochastic
pde (such as $u$ in \eqref{sm}). In our case, $\mathcal{F}$ is the
$\sigma$-algebra generated by the Wiener process $W:=\big\{
W(t,x):\; t\in[0,T],\; x\in{\mathcal{D}}\big\}$, and the relevant
topological structure is this of the Hilbert space
$L^2([0,T]\times \mathcal{D})$.

\subsection{Main Results}
We investigate if $u$, the solution of \eqref{sm}, as a random
variable, has a density; an affirmative answer is given by proving
that the law of $u$ is absolutely continuous.

Here, we follow the strategy proposed by Cardon-Weber in
\cite{CW1}, and approximate $u$ by a sequence $u_n$ for which we
prove existence of Malliavin derivative; we then check that a
certain norm of this derivative is almost surely strictly
positive. Strict positivity establishes the absolute continuity
of the sequence $u_n$ and on the limit, as $n\rightarrow\infty$,
the same result follows for $u$, cf. Subsections 3.1, 3.2.

We use carefully some important definitions and results from the
theory of Malliavin Calculus, presented by Nualart in \cite{N}.

More precisely, in dimension $d= 1$, we show that the stochastic
solution is locally differentiable in the sense of Malliavin
calculus. Under some non-degeneracy condition on the noise
diffusion coefficient, we prove that the law of the solution is
absolutely continuous with respect to the Lebesgue measure
$\mathbb{R}$.

Cardon-Weber in \cite{CW1} studied the stochastic Cahn-Hilliard
equation with bounded noise diffusion. In our case we consider a
more general problem; this of the stochastic
Cahn-Hilliard/Allen-Cahn equation with unbounded noise diffusion,
for which when $d=1$ in \cite{AKM}, the authors established
existence of a continuous solution a.s. This equation contains a
new second order nonlinear operator, fact that arises the use of
a new spde, quite different than this proposed in \cite{CW1},
which defines a proper approximating sequence $u_n$.
Additionally, we treat efficiently the existing growth of the
unbounded diffusion, by proving estimates in expectation in the
stronger $L^\infty(\mathcal{D})$-norm, in various places,
involving $u_n$ and its Malliavin derivative.

The novelty of this paper is the proof of Theorem \ref{mres}
(i.e., Theorem \ref{loc} and Theorem \ref{denu}), for the equation
\eqref{sm}, which consists a stochastic pde with a white
space-time noise and unbounded noise diffusion. This is an
important contribution to the literature of stochastic equations
stemming from physical problems, such as phase separation in the
presence of randomness. Our result is set in the very active area
of research on well posedness (existence and regularity) of
solutions of spdes. Moreover, these solutions are random
variables depending not only on space and time but also on the
parameter $\omega\in\Omega$. Hence, by proving that a density
exists for $u$, we integrate significantly the theoretical
analysis of this stochastic model.
\medskip

In particular, we prove the next Main Theorem.
\begin{theorem}\label{mres}
Let $u$ be the solution of the stochastic Cahn-Hilliard/Allen-Cahn equation \eqref{sm}, with the Neumann b.c. \eqref{neu}
in dimension $d=1$, for $\mathcal{D}:=(0,\pi)$, with smooth
initial condition $u_0$.

Let the noise diffusion $\sigma$ satisfy:
\begin{enumerate}
\item
$\sigma$ has a sublinear growth uniformly for any $x\in\mathbb{R}$
of the form
\begin{equation}\label{LEq}
|\sigma(x)|\leq C(1+|x|^q),
\end{equation}
for some $C >0$ and any $q \in (0,\frac{1}{3})$,
\item
$\sigma$ is Lipschitz on $\mathbb{R}$, i.e., there exists $K$:
\begin{equation}\label{s3}
|\sigma(x)-\sigma(y)|\leq K|x-y|,\;\;\forall\;x,\;y\in\mathbb{R},
\end{equation}
\item
$\sigma$ is continuously differentiable on $\mathbb{R}$ (i.e.,
$\exists\;\sigma'$, and $\sigma,\;\sigma'$ are continuous), and
since $\sigma'$ exists, due to \eqref{s3} it follows that
\begin{equation}\label{s4}
|\sigma'(x)|\leq K,\;\;\forall\;x\in\mathbb{R}.
\end{equation}
\end{enumerate}
Then the derivative of $u$ in the Malliavin sense exists locally
(cf. Theorem \ref{loc}).

Moreover, if, in addition, $\sigma$ is non-degenerate, i.e., there
exists $c_0>0$ such that
\begin{equation}\label{s1}
|\sigma(x)|\geq c_0>0,\;\;\forall\;x\in\mathbb{R},
\end{equation}
then the law of $u$ is absolutely continuous with respect to the
Lebesgue measure $\mathbb{R}$ (cf. Theorem \ref{denu}).
\end{theorem}
\begin{remark}\label{remch}
The above theorem is also valid in the more general case of
\begin{equation} \label{smv}
u_t=-\varrho\Delta\Big{(}\Delta u-f(u)\Big{)}
+\tilde{q}\Big{(}\Delta u-f(u)\Big{)}+\sigma(u)
\dot{W},\;\;\;\;t>0,\;x\in\mathcal{D},
\end{equation}
for $\varrho>0$ and $\tilde{q}\geq 0$, cf. Section 4 of \cite{AKM}
for the relevant discussion for the existence and regularity of
solution for this more general problem, and the observations for
the Green's function. In our case, when establishing existence of
a density, all our results hold true for \eqref{smv} also.

Thus, for $\varrho=1$, $\tilde{q}:=0$, the Main Theorem
\ref{mres} (existence of Malliavin derivative locally and of a
density for $u$) is valid for the one-dimensional stochastic
Cahn-Hilliard equation with unbounded noise diffusion and
non-smooth in space and in time space-time noise.
\end{remark}

The structure of the rest of this paper is as follows: Section 2
presents some basic definitions from Malliavin calculus such as
the definitions of the spaces of random variables $D^{1,2}$,
$L^{1,2}$, and their local versions $D^{1,2}_{\rm loc}$,
$L^{1,2}_{\rm loc}$. Moreover, due to the fact that $u$ is a.s.
continuous, we are able to approximate efficiently the solution
$u$ by some $u_n$ defined through an spde, for which we prove
existence of the Malliavin derivative; $u_n$ is proven to be a
localization in the Malliavin sense of $u$, which yields finally
the existence of the Malliavin derivative of $u$ locally. In
details, $u$ is written in the integral representation given by
\eqref{original}. This representation motivates the piece-wise
approximation $u_n$ definition as the solution of the spde
\eqref{piecewise ff}. Lemma \ref{lemma 2.1} establishes existence
and uniqueness of $u_n$, and provides a useful bound in
expectation. We then prove that the Malliavin derivative of $u_n$
is well defined, and that $u_n\in D^{1,2}$, cf. Proposition
\ref{LL234}; a direct consequence is the Main Theorem \ref{loc},
i.e., that $u$ belongs to $L^{1,2}_{\rm loc}\subseteq D^{1,2}_{\rm
loc}$.

In Section 3, we prove the absolute continuity of the
approximations $u_n$ which again through a localization argument
(see Remark \ref{locderiv}) yields the existence of a density for
the stochastic solution $u$. More specifically, we present first
the very technical Lemma \ref{leps}, where the growth of the
unbounded noise diffusion $\sigma$ is crucial. In the sequel,
under the additional assumption \eqref{s1} (non-degenerating
$\sigma$), we establish, in Theorem \ref{denu0}, the absolute
continuity of $u_n$, and thus, the existence of a density for $u$
(Main Theorem \ref{denu}).

For the rest of this paper, we consider $d=1$,
$\mathcal{D}:=(0,\pi)$, a smooth $u_0$, and the assumptions (1),
(2), (3), of the statement of Theorem \ref{mres}, for the
diffusion $\sigma$; for simplicity, we set in \eqref{sm}
$\varrho:=1$. The additional assumption \eqref{s1} for a
non-degenerate $\sigma$ appears only in the statements (and
proofs) of Theorems \ref{denu0}, \ref{denu}.
\section{Malliavin calculus}
\subsection{Basic definitions}
\begin{definition}\label{def1}
Following the notation of \cite{CW1}, we denote by $D^{1,2}$ the
set of random variables $v$ such that the Malliavin derivative
(in space and time) $D_{y,s}v(x,t)$ exists, for any
$y\in\mathcal{D}$ and any $s\geq 0$ and any
$(x,t)\in\mathcal{D}\times [0,T]$, and satisfies
\begin{equation}\label{m1}
\|v\|_{D^{1,2}}:=\Big{(}\mathbf{E}(|v|^2)+\mathbf{E}(\|D_{\cdot,\cdot}v\|_{L^2([0,T]\times\mathcal{D})}^2)\Big{)}^{1/2}<\infty,
\end{equation}
for
$$\|D_{\cdot,\cdot}v(x,t)\|_{L^2([0,T]\times\mathcal{D})}:=\Big{(}\int_0^T\int_\mathcal{D}|D_{y,s}v(x,t)|^2dyds\Big{)}^{1/2}.$$

Indeed, according to \cite{N}, p. 27 (where the definition of
$D^{1,p}$, $p\geq 1$, is given), $D^{1,2}$ is a Hilbert space and
consists the closure of the class of smooth random variables $v$
in the norm
$$\|v\|_{D^{1,2}}:=\Big{(}\mathbf{E}(|v|^2)+\mathbf{E}(\|D_{\cdot,\cdot}v\|_{H}^2)\Big{)}^{1/2},$$
where $\|\cdot\|_{H}$ is the norm induced by the inner product
$<\cdot,\cdot>_H$ and the norm $\|\cdot\|_{D^{1,2}}$ is induced by
the inner product
$$<f,g>:=\mathbf{E}(fg)+\mathbf{E}(<D_{\cdot,\cdot}f,D_{\cdot,\cdot}g>_H),$$
where, in our case, $H:=L^2([0,T]\times\mathcal{D})$ and
$<\cdot,\cdot>_H$ is the usual $L^2$ inner product on
$[0,T]\times\mathcal{D}$.
\end{definition}

\begin{definition}\label{def2}
The set $L^{1,2}$ is defined as the class of all stochastic
processes $v=v(x,t)\in L^2(\Omega\times[0,T]\times\mathcal{D})$,
i.e.,
$$\|v\|_{L^2(\Omega\times[0,T]\times\mathcal{D})}:=
\Big{(}\mathbf{E}\Big{(}\int_0^T\int_{\mathcal{D}}|v(x,t)|^2dxdt\Big{)}\Big{)}^{1/2}<\infty,$$
such that $v\in D^{1,2}$ and satisfy
$$\mathbf{E}\Big{(}\int_0^T\int_{\mathcal{D}}\int_0^T\int_{\mathcal{D}}|D_{y,s}v(x,t)|^2dydsdxdt\Big{)}<\infty,$$
cf. \cite{N}, p. 42 (to avoid any confusion, we point out that
the notation $T$ used by Nualart at p. 42, in our case
corresponds to $[0,T]\times\mathcal{D}$), and \cite{CW1}.
\end{definition}
\begin{definition}\label{def3}
According to \cite{N}, p. 49, for $L:=L^{1,2}$ (a class of
stochastic processes), $(L^{1,2})_{\rm loc}=:L_{\rm loc}^{1,2}$
is defined as the set of random variables $v$: $\exists$ a
sequence $\{(\Omega_n,v_n),\;n\geq 1\}\subset\mathcal{F}\times
L\;\;({\rm here\; }L:=L^{1,2})$ such that
\begin{enumerate}
\item $\Omega_n\uparrow\Omega$ a.s.,
\item $v=v_n$ a.s. on $\Omega_n$.
\end{enumerate}
Also, for $L:=D^{1,2}$ (a class of random variables),
$(D^{1,2})_{\rm loc}=:D_{\rm loc}^{1,2}$ is defined as the set of
random variables $v$: $\exists$ a sequence
$\{(\Omega_n,v_n),\;n\geq 1\}\subset\mathcal{F}\times L\;\;({\rm
here \;}L:=D^{1,2})$ such that
\begin{enumerate}
\item $\Omega_n\uparrow\Omega$ a.s.,
\item $v=v_n$ a.s. on $\Omega_n$.
\end{enumerate}
Here, $\mathcal{F}$ is the $\sigma$-algebra, while
$\Omega_n\uparrow\Omega$  a.s., is equivalent to
$\Omega_1\subseteq
\Omega_2\subseteq\Omega_3\subseteq\cdots\subseteq \Omega$, such
that
$$\displaystyle{\lim_{n\rightarrow\infty}}P(\Omega_n)=P(\Omega)=1.$$
\end{definition}
\begin{remark}\label{remd}
If $v\in D^{1,2}_{\rm loc}$ and $(\Omega_n,v_n)$ localizes $v$ in
$D^{1,2}$ in the aforementioned way (cf. the previous definition),
then the Malliavin derivative $D_{y,s}v$ is defined without
ambiguity by $D_{y,s}v=D_{y,s}v_n$ on $\Omega_n$ for $n\geq 1$
(i.e., $D_{y,s}v$ is well defined by localization in the space
$D^{1,2}_{\rm loc}$), cf. \cite{N}, p. 49.

\end{remark}
\subsection{Localization of $u$ in $L^{1,2}$}
Our aim is to prove that the stochastic solution $u$ of
\eqref{sm} belongs to the space $L_{\rm loc}^{1,2}$, (observe
that $L_{\rm loc}^{1,2}\subseteq D_{\rm loc}^{1,2}$).

\begin{remark}\label{remnew1}
Note, that $L^{1,2}$ is a subset of $D^{1,2}$, consisting of more
regular random variables in $L^2(\Omega\times
[0,T]\times\mathcal{D})$, with Malliavin derivative bounded in
$L^2(\Omega\times ([0,T]\times\mathcal{D})^2)$. Hence, a
constructed localization $(\Omega_n,u_n)$ of $u$ in $L^{1,2}$ is
also a localization in $D^{1,2}$ and thus, will define well the
Malliavin derivative of $u$ through the Malliavin derivative of
$u_n$ (see Remark \ref{remd}). Moreover, the previous
construction, will establish local regularity of the solution $u$
of \eqref{sm} in the sense of Malliavin calculus.
\end{remark}

The solution $u$ of the stochastic equation \eqref{sm} is written
in integral representation as
\begin{align} \label{original}
u(x,t) =& \int_{\mathcal{D}}u_0(y)G(x,y,t) dy &&\nonumber \\
&+ \int_{0}^{t}\int_{\mathcal{D}}[\Delta G(x,y,t-s)-G(x,y,t-s)]f(u(y,s)) dy ds &&\nonumber \\
&+ \int_{0}^{t}\int_{\mathcal{D}}G(x,y,t-s)\sigma(u(y,s))W(dy,ds),
\end{align}
for,
\begin{equation} \label{KJ2211}
G(x,y,t):=\sum_{k=0}^\infty
e^{-({\lambda^2_k+\lambda_k})t}\alpha_k(x)\alpha_k(y),
\end{equation}
where $\lambda_k$ are the eigenvalues of the negative Neumann
Laplacian with Neumann b.c. posed on $\mathcal{D}$, and
$\{a_k\}_{k\in\mathbb{N}}$ a corresponding eigenfunction
orthonormal basis of $L^2(\mathcal{D})$; see \cite{AKM} for more
details on \eqref{original} and the definition of Green's
function $G$.
\subsubsection{Piece-wise approximation of the stochastic solution}

We shall construct a 'piecewise' approximation $u_n\in L^{1,2}$
of $u$.

Let $H_n:\mathbb{R}^+\rightarrow\mathbb{R}^+$ be a $\mathcal{C}^1$
cut-off function satisfying
$$|H_n|\le 1,\mbox{ and }|H^\prime_n|\le 2,$$ for any $n>0$, with
\begin{equation}
H_n(x):=
\begin{cases}
1  &     \textrm{if  $|x|< n $ }   \\
0 & \textrm{if $|x|\ge n+1 $}.
\end{cases}
\end{equation}
We set
\begin{equation} \label{DF11}
 f_n(x):=H_n(|x|)f(x),
\end{equation}
obviously $f_n$ is a $\mathcal{C}^1$ function and its derivative
is bounded, \cite{W}; this bound depends on $n$ and consists a
Lipschitz coefficient for $f_n$.

We define
$$\Omega_{n}:= \Big{\{} \omega\in\Omega:\;\sup_{t \in [0,T]}\sup_{x \in \mathcal{D}} |u(x,t;\omega)| < n
\Big{\}}.$$ Obviously, it holds that
$$\Omega_1\subseteq\Omega_2\subseteq\cdots\subseteq \Omega.$$
Let
\begin{align} \label{piecewise ff}
u_n(x,t) :=&\int_{\mathcal{D}}u_0(y)G(x,y,t) dy &&\nonumber \\
&+ \int_{0}^{t}\int_{\mathcal{D}}[\Delta G(x,y,t-s)-G(x,y,t-s)]f_n(u_n(y,s)) dy ds &&\nonumber \\
&+
\int_{0}^{t}\int_{\mathcal{D}}G(x,y,t-s)\sigma(u_n(y,s))W(dy,ds).
\end{align}
We shall prove existence and uniqueness of solution $u_n$ of
$(\ref{piecewise ff})$, and we shall establish that $u_n$ belongs
in the space $L^{1,2}$; this will yield that the solution $u$ is
in the space $L_{\rm loc}^{1,2}$.

We assume that the initial condition $u_0$ is smooth; according
to \cite{AKM}, in dimensions $d=1$, due to the stated at the
introduction assumptions for $\sigma$, in particular the
Lipschitz property and the growth of order $q<\frac{1}{3}$ (in
\cite{AKM}, $\sigma$ is just Lipschitz with sublinear growth of
order $q<\frac{1}{3}$ and not assumed also continuously
differentiable or non-degenerate), the solution $u$ of \eqref{sm}
exists and is a.s. continuous. We need a.s. continuity of $u$ in
order to establish our arguments, and this is the main reason why
our Main Result is restricted in dimensions $d=1$. More
precisely, the a.s. continuity of $u$ yields, cf. also in
\cite{CW1}
$$P(\Omega_n)\rightarrow 1\;\;\mbox{as}\;n\rightarrow\infty,$$
which is needed for the definition of the localization of $u$.

The rest of this paragraph, will be devoted to the proof of the
next, quite technical lemma, which establishes the existence of
the piece-wise approximation $u_n$, and provides a useful bound in
expectation.

\begin{lemma} \label{lemma 2.1}
The problem (\ref{piecewise ff}) has a unique solution $u_n$, in
dimensions $d=1$.

Moreover, $u_n$ satisfies for any $p\geq 2$
\begin{equation}\label{feq}
\sup_{t\in[0,T]}\mathbf{E}
(\|u_n(\cdot,t)\|_{L^\infty(\mathcal{D})}^p)<\infty.
\end{equation}
\end{lemma}
\begin{proof}
The basic idea is the construction of a Cauchy sequence, through a
Picard iteration scheme, which converges, at a certain norm, to
the solution $u_n$ of
 (\ref{piecewise ff}).

For given $n$, we define
$$u_{n,0}(x,t):=G_tu_0(x),$$
and for any integer $k\geq 0$, we consider the following Picard
iteration scheme, which is motivated by $(\ref{piecewise ff})$,
\begin{align} \label{Picard iteration k+1}
u_{n,k+1}(x,t) =& \int_{\mathcal{D}}u_0(y)G(x,y,t) dy &&\nonumber \\
&+ \int_{0}^{t}\int_{\mathcal{D}}[\Delta G(x,y,t-s)-G(x,y,t-s)] f_n(u_{n,k}(y,s)) dy ds &&\nonumber \\
&+
\int_{0}^{t}\int_{\mathcal{D}}G(x,y,t-s)\sigma(u_{n,k}(y,s))W(dy,ds).
\end{align}


Relation $(\ref{Picard iteration k+1})$ yields for any $k \ge 1$,
\begin{equation} \label{Picard iteration combination}
\begin{split}
u_{n,k+1}(x,t)-u_{n,k}(x,t)=& \int_{0}^{t}\int_{\mathcal{D}}[\Delta G(x,y,t-s)-G(x,y,t-s)] \Big ( f_n(u_{n,k}(y,s))-f_n(u_{n,k-1}(y,s)) \Big )dyds  \\
 &+ \int_{0}^{t}\int_{\mathcal{D}}G(x,y,t-s) \Big (\sigma(u_{n,k}(y,s))-  \sigma(u_{n,k-1}(y,s))\Big
 )W(dy,ds).
\end{split}
\end{equation}
Hence, we obtain
\begin{equation*}
\begin{split}
|u_{n,k+1}&(x,t)-u_{n,k}(x,t)|
\le \int_{0}^{t}\int_{\mathcal{D}}|\Delta G(x,y,t-s)|   | f_n(u_{n,k} (y,s))-f_n(u_{n,k-1}(y,s)) | dyds \\
&+\int_{0}^{t}\int_{\mathcal{D}}|G(x,y,t-s)|   | f_n(u_{n,k}(y,s))-f_n(u_{n,k-1}(y,s)) | dyds    \\
 &+ \Big |\int_{0}^{t}\int_{\mathcal{D}} G(x,y,t-s)(\sigma(u_{n,k}(y,s))-  \sigma(u_{n,k-1}(y,s)))W(dy,ds) \Big
 |.
\end{split}
\end{equation*}
Thus, taking $p$ powers for $p\geq 2$, and then supremum for any
$x\in\mathcal{D}$ and supremum in time for the stochastic
integral, and then expectation, we get
\begin{equation*}
\begin{split}
\mathbf{E}
(\sup_{x\in\mathcal{D}}&|u_{n,k+1}(x,t)-u_{n,k}(x,t)|^p)=\mathbf{E}
(\|u_{n,k+1}(\cdot,t)-u_{n,k}(\cdot,t)\|_{L^\infty(\mathcal{D})}^p)\\
 \le&
c\mathbf{E} \bigg (\sup_{x\in\mathcal{D}}
\Big{(}\int_{0}^{t}\int_{\mathcal{D}}|\Delta G(x,y,t-s)|
  | f_n(u_{n,k}(y,s))-f_n(u_{n,k-1}(y,s)) | dyds \Big )^p \bigg )  \\
&+ c\mathbf{E} \bigg (\sup_{x\in\mathcal{D}} \Big
(\int_{0}^{t}\int_{\mathcal{D}}|G(x,y,t-s)|
   | f_n(u_{n,k}(y,s))-f_n(u_{n,k-1}(y,s)) | dyds \Big )^p  \bigg ) \\
&+c\mathbf{E} \bigg (\sup_{\tau\in[0,t]}\sup_{x\in\mathcal{D}}
\Big |\int_{0}^{\tau}\int_{\mathcal{D}} G(x,y,\tau-s)
[\sigma(u_{n,k}(y,s))- \sigma(u_{n,k-1}(y,s)) ] W(dy,ds) \Big |^p
\bigg ).
\end{split}
\end{equation*}

The function $f_n$ is Lipschitz and so,
\begin{equation*}
\begin{split}
\mathbf{E}
(\|u_{n,k+1}(\cdot,t)-&u_{n,k}(\cdot,t)\|_{L^\infty(\mathcal{D})}^p) \\
 \le&
c\mathbf{E} \bigg ( \sup_{x\in\mathcal{D}}\Big
(\int_{0}^{t}\int_{\mathcal{D}}|\Delta G(x,y,t-s)|
   | u_{n,k}(y,s)-u_{n,k-1}(y,s) | dyds \Big )^p \bigg )  \\
&+ c\mathbf{E} \bigg (\sup_{x\in\mathcal{D}} \Big
(\int_{0}^{t}\int_{\mathcal{D}}|G(x,y,t-s)|
   | u_{n,k}(y,s)-u_{n,k-1}(y,s) | dyds \Big )^p  \bigg ) \\
&+c\mathbf{E} \bigg (\sup_{\tau\in [0,t]}
\sup_{x\in\mathcal{D}}\Big |\int_{0}^{\tau}\int_{\mathcal{D}}
G(x,y,\tau-s) [\sigma(u_{n,k}(y,s))-  \sigma(u_{n,k-1}(y,s)) ]
W(dy,ds) \Big |^p \bigg ).
\end{split}
\end{equation*}
Burkholder-Davis-Gundy inequality applied to the stochastic term
of the previous inequality gives
\begin{equation}\label{d1n}
\begin{split}
\mathbf{E}
\Big{(}\|u_{n,k+1}(\cdot,t)-&u_{n,k}(\cdot,t)\|_{L^\infty(\mathcal{D})}^p\Big{)}\\
\le &c\mathbf{E} \bigg ( \sup_{x\in\mathcal{D}}\Big
(\int_{0}^{t}\int_{\mathcal{D}}|\Delta G(x,y,t-s)|
  | u_{n,k}(y,s)-u_{n,k-1}(y,s) | dyds \Big )^p \bigg )  \\
&+ c\mathbf{E} \bigg ( \sup_{x\in\mathcal{D}}\Big
(\int_{0}^{t}\int_{\mathcal{D}}|G(x,y,t-s)|
 | u_{n,k}(y,s)-u_{n,k-1}(y,s) | dyds \Big )^p  \bigg ) \\
& +c\mathbf{E} \bigg ( \sup_{\tau\in [0,t]}
\sup_{x\in\mathcal{D}}\Big (\int_{0}^{\tau}\int_{\mathcal{D}}|
G(x,y,\tau-s)|^2 |\sigma(u_{n,k}(y,s))-\sigma(
u_{n,k-1}(y,s)) |^2 dyds \Big )^{p/2} \bigg )\\
\le &c\mathbf{E} \bigg (\sup_{x\in\mathcal{D}} \Big (\int_{0}^{t}
\int_{\mathcal{D}}|\Delta G(x,y,t-s)|   | u_{n,k}(y,s)-u_{n,k-1}(y,s) | dyds \Big )^p \bigg )  \\
&+ c\mathbf{E} \bigg ( \sup_{x\in\mathcal{D}}\Big (\int_{0}^{t}
\int_{\mathcal{D}}|G(x,y,t-s)|   | u_{n,k}(y,s)-u_{n,k-1}(y,s) | dyds \Big )^p  \bigg ) \\
& +c\mathbf{E} \bigg
(\sup_{\tau\in[0,t]}\sup_{x\in\mathcal{D}}\Big
(\int_{0}^{\tau}\int_{\mathcal{D}}| G(x,y,\tau-s)|^2
|u_{n,k}(y,s)- u_{n,k-1}(y,s) |^2 dyds \Big )^{p/2} \bigg ),
\end{split}
\end{equation}
where for the last inequality we used that the diffusion
coefficient $\sigma$ is Lipschitz, uniformly for any $n$.

Thus, from \eqref{d1n}, we derived
\begin{equation} \label{le1.6}
\mathbf{E}
\Big{(}\|u_{n,k+1}(\cdot,t)-u_{n,k}(\cdot,t)\|_{L^\infty(\mathcal{D})}^p\Big{)}
\le Q_1(t)+Q_2(t)+Q_3(t),
\end{equation}
where
\begin{equation*}
\begin{split}
Q_1(t):=&c\mathbf{E} \bigg (\Big{\|}\int_{0}^{t}\int_{\mathcal{D}}
|\Delta G(\cdot,y,t-s)|   | u_{n,k}(y,s)-u_{n,k-1}(y,s) | dyds \Big \|_{L^\infty(\mathcal{D})}^p \bigg ), \\
Q_2(t):=&c\mathbf{E} \bigg (\Big \|\int_{0}^{t}
\int_{\mathcal{D}}|G(\cdot,y,t-s)|   | u_{n,k}(y,s)-u_{n,k-1}(y,s) | dyds \Big \|_{L^\infty(\mathcal{D})}^p  \bigg ), \\
Q_3(t):=&c\mathbf{E} \bigg ( \sup_{\tau\in[0,t]}\Big
\|\int_{0}^{\tau}\int_{\mathcal{D}}| G(\cdot,y,\tau-s)|^2
|u_{n,k}(y,s)- u_{n,k-1}(y,s) |^2 dyds \Big
\|_{L^\infty(\mathcal{D})}^{p/2} \bigg ).
\end{split}
\end{equation*}

In the sequel, we shall estimate the terms involving the Green's
function $G$ by using Lemma 1.6 of \cite{CW1} for
$\rho=q=\infty$, $r=1$ (which holds true when $H$ and $\upsilon$
are replaced by their absolute values, cf. the proof of lemma
presented in \cite{CW1}).

For estimating the term $Q_1(t)$, we choose the inequality (1.12)
of \cite{CW1}, p. 781, for $$H(x,y,t-s):=|\Delta
G(x,y,t-s)|,\;\;\;\upsilon( y ,s):=|
u_{n,k}(y,s)-u_{n,k-1}(y,s)|,\;\;\;\rho=q=\infty,\;\;\;r=1.$$

Thus, we have for $p>2$
\begin{equation}\label{da1}
\begin{split}
Q_1(t) \le &c\mathbf{E} \bigg ( \Big |\int_{0}^{t}
\frac{1}{(t-s)^{(d+2)/4-d/4}}
 \| u_{n,k}(\cdot,s)-u_{n,k-1}(\cdot,s) \|_{L^\infty(\mathcal{D})} ds \Big |^p \bigg ) \\
\le &c\mathbf{E} \bigg ( \Big |\int_{0}^{t}
\frac{1}{(t-s)^{\frac{1}{2}}}   \|
u_{n,k}(\cdot,s)-u_{n,k-1}(\cdot,s) \|_{L^\infty(\mathcal{D})} ds
\Big |^p \bigg )\\
\le &c\mathbf{E} \bigg ( \Big (\int_{0}^{t}
\frac{1}{(t-s)^{q\frac{1}{2}}}ds\Big )^{p/q}\Big (\int_0^t \|
u_{n,k}(\cdot,s)-u_{n,k-1}(\cdot,s) \|^p_{L^\infty(\mathcal{D})}
ds \Big )^{p/p} \bigg )\\
=&c\mathbf{E} \bigg ( \Big (\int_{0}^{t}
\frac{1}{(t-s)^{q\frac{1}{2}}}ds\Big )^{p/q}\int_0^t \|
u_{n,k}(\cdot,s)-u_{n,k-1}(\cdot,s) \|^p_{L^\infty(\mathcal{D})}
ds\bigg )\\
\le&c\mathbf{E} \bigg (\int_0^t \|
u_{n,k}(\cdot,s)-u_{n,k-1}(\cdot,s) \|^p_{L^\infty(\mathcal{D})}
ds \bigg ),
\end{split}
\end{equation}
where we used H\"older inequality for $q=p/(p-1)$, i.e.,
$1/p+1/q=1/p+(p-1)/p=1$, and the fact that $0<\frac{q}{2}<1$ or
equivalently $0<\frac{p}{p-1}<2$, which is true for any $p>2$, and
thus
$$\Big (\int_{0}^{t} \frac{1}{(t-s)^{q\frac{1}{2}}}ds\Big )^{p/q}<c.$$


For the term $Q_2(t)$ we choose the inequality (1.11) of
\cite{CW1}, p. 781, for $$H(x,y,t-s)= |G(x,y,t-s)|,\;\;\;
\upsilon( y,s)=|
u_{n,k}(y,s)-u_{n,k-1}(y,s)|,\;\;\;\rho=q=\infty,\;\;\;r=1.$$ Then
we get
\begin{equation}\label{da2}
\begin{split}
Q_2(t) \le &c\mathbf{E} \bigg (\Big |\int_{0}^{t}
\frac{1}{(t-s)^{(\frac{d}{4})(1-1)}}
 \| u_{n,k}(\cdot,s)-u_{n,k-1}(\cdot,s) \|_{L^\infty(\mathcal{D})} ds \Big |^p \bigg ) \\
= &c\mathbf{E} \bigg (\Big |\int_{0}^{t}   \|
u_{n,k}(\cdot,s)-u_{n,k-1}(\cdot,s) \|_{L^\infty(\mathcal{D})} ds
\Big |^p \bigg )\\
\le &c \mathbf{E} \bigg (\Big (\int_{0}^{t}  1^qds\Big )^{p/q}\Big
(\int_0^t \| u_{n,k}(\cdot,s)-u_{n,k-1}(\cdot,s)
\|^p_{L^\infty(\mathcal{D})} ds \Big )^{p/p} \bigg )\\
\le &c\mathbf{E} \bigg (\int_0^t \|
u_{n,k}(\cdot,s)-u_{n,k-1}(\cdot,s) \|^p_{L^\infty(\mathcal{D})}
ds \bigg ),
\end{split}
\end{equation}
where we used H\"older inequality for $q=p/(p-1)$.


For the term $Q_3(t)$ we choose the inequality (1.13) of
\cite{CW1}, p. 781, for $$H(x,y,t-s)= G^2(x,y,t-s),\;\;\;
\upsilon(y
,s)=|u_{n,k}(y,s)-u_{n,k-1}(y,s)|^2,\;\;\;\rho=q=\infty,\;\;\;r=1,$$
and we obtain
\begin{equation}\label{da3}
\begin{split}
Q_3(t) \le &c\mathbf{E} \bigg (\sup_{\tau\in[0,t]} \Big
|\int_{0}^{\tau} \frac{1}{(\tau-s)^{\frac{d}{2}-\frac{d}{4}}}
  \| |u_{n,k}(\cdot,s)-u_{n,k-1}(\cdot,s)|^2 \|_{L^\infty(\mathcal{D})} ds \Big |^{p/2} \bigg ) \\
= &c\mathbf{E} \bigg (\sup_{\tau\in[0,t]} \Big |\int_{0}^{\tau}
\frac{1}{(\tau-s)^{\frac{d}{4}}}   \|
u_{n,k}(\cdot,s)-u_{n,k-1}(\cdot,s) \|^2_{L^\infty(\mathcal{D})}
ds \Big |^{p/2} \bigg )\\
\le &c\mathbf{E} \bigg ( \sup_{\tau\in[0,t]}\Big[\Big
(\int_{0}^{\tau}
\frac{1}{(\tau-s)^{q\frac{d}{4}}}ds\Big)^{\frac{p}{2q}}
\Big[\int_0^\tau\| u_{n,k}(\cdot,s)-u_{n,k-1}(\cdot,s)
\|^{2p/2}_{L^\infty(\mathcal{D})} ds \Big ]^{\frac{p/2}{p/2}}\Big]
\bigg )\\
\le &c\mathbf{E} \bigg ( \sup_{\tau\in[0,t]}\Big[\Big
(\int_{0}^{\tau}
\frac{1}{(\tau-s)^{q\frac{d}{4}}}ds\Big)^{\frac{p}{2q}}\Big]
\sup_{\tau\in[0,t]}\Big[\int_0^\tau\|
u_{n,k}(\cdot,s)-u_{n,k-1}(\cdot,s) \|^p_{L^\infty(\mathcal{D})}
ds \Big ]\bigg )\\
\le &c\mathbf{E} \bigg (\int_0^t\|
u_{n,k}(\cdot,s)-u_{n,k-1}(\cdot,s) \|^p_{L^\infty(\mathcal{D})}
ds \bigg ),
\end{split}
\end{equation}
where we used H\"older inequality for $1/q+1/(p/2)=(p-2)/p+2/p=1$,
i.e., for $q=p/(p-2)$ which gives
$0<q\frac{d}{4}=\frac{p}{p-2}\frac{d}{4}<1$ true for
$p>\frac{8}{4-d}$ and thus,
$$\Big(\int_{0}^{\tau}\frac{1}{(\tau-s)^{q\frac{d}{4}}}ds\Big)^{\frac{p}{2q}}<\infty.$$

Replacing \eqref{da1}, \eqref{da2} and \eqref{da3} to
\eqref{le1.6}, we obtain for any $p>\max\{2,8/(4-d)\}$ and any
integer $k\geq 1$
\begin{equation}\label{d55}
\begin{split}
\mathbf{E}
(\|u_{n,k+1}(\cdot,t)-u_{n,k}(\cdot,t)\|_{L^\infty(\mathcal{D})}^p)
&\le c\mathbf{E} \bigg (\int_{0}^{t} \|
u_{n,k}(\cdot,s)-u_{n,k-1}(\cdot,s)
\|^p_{L^\infty(\mathcal{D})}ds \bigg )\\
&\le c\int_{0}^{t}\mathbf{E}( \|
u_{n,k}(\cdot,s_k)-u_{n,k-1}(\cdot,s_k)
\|^p_{L^\infty(\mathcal{D})})ds_k,
\end{split}
\end{equation}
where we used Fubini's Theorem.

Inequality \eqref{d55} applied for the term $\mathbf{E} (\|
u_{n,k}(\cdot,s_k)-u_{n,k-1}(\cdot,s_k)
\|^p_{L^\infty(\mathcal{D})})$ gives
\begin{equation*}
\begin{split}
\mathbf{E}
(\|u_{n,k+1}(\cdot,t)-u_{n,k}(\cdot,t)\|_{L^\infty(\mathcal{D})}^p)
\le c^2\int_{0}^{t}\int_0^{s_k}\mathbf{E}( \|
u_{n,k-1}(\cdot,s_{k-1})-u_{n,k-2}(\cdot,s_{k-1})
\|^p_{L^\infty(\mathcal{D})})ds_{k-1}ds_k,
\end{split}
\end{equation*}
i.e., we get
\begin{equation}\label{da77}
\begin{split}
\mathbf{E}
(\|u_{n,k+1}&(\cdot,t)-u_{n,k}(\cdot,t)\|_{L^\infty(\mathcal{D})}^p)
\le c^2\int_{0}^{t}\int_0^{s_k}\mathbf{E}( \|
u_{n,k-1}(\cdot,s_{k-1})-u_{n,k-2}(\cdot,s_{k-1})
\|^p_{L^\infty(\mathcal{D})})ds_{k-1}ds_k\\
\le& c^3\int_{0}^{t}\int_0^{s_k}\int_0^{s_{k-1}}\mathbf{E} (\|
u_{n,k-2}(\cdot,s_{k-2})-u_{n,k-3}(\cdot,s_{k-2})
\|^p_{L^\infty(\mathcal{D})})ds_{k-2}ds_{k-1}ds_k
\le\cdots\\
\le&c^{k}\int_{0}^{t}\int_0^{s_k}\int_0^{s_{k-1}}\cdots\int_0^{s_2}\mathbf{E}
(\| u_{n,1}(\cdot,s_{1})-u_{n,0}(\cdot,s_{1})
\|^p_{L^\infty(\mathcal{D})})d_{s_1}\cdots ds_{k-2}ds_{k-1}ds_k
\\
\le
&c^{k}\int_{0}^{t}\int_0^{s_k}\int_0^{s_{k-1}}\cdots\int_0^{s_2}1d_{s_1}\cdots
ds_{k-2}ds_{k-1}ds_k \sup_{t\in[0,T]}\mathbf{E} (\|
u_{n,1}(\cdot,t)-u_{n,0}(\cdot,t) \|^p_{L^\infty(\mathcal{D})})\le
c\frac{c^{2k}}{k!},
\end{split}
\end{equation}
for any $t\in[0,T]$, where we applied the next calculation
\begin{equation*}
\begin{split}
\int_{0}^{t}&\int_0^{s_k}\int_0^{s_{k-1}}\cdots\int_0^{s_2}1d_{s_1}\cdots
ds_{k-2}ds_{k-1}ds_k=\int_{0}^{t}\int_0^{s_k}\int_0^{s_{k-1}}\cdots\int_0^{s_3}s_2d_{s_2}\cdots
ds_{k-2}ds_{k-1}ds_k\\
=&\int_{0}^{t}\int_0^{s_k}\int_0^{s_{k-1}}\cdots\int_0^{s_4}\frac{s_3^2}{2}d_{s_3}\cdots
ds_{k-2}ds_{k-1}ds_k=\int_{0}^{t}\int_0^{s_k}\int_0^{s_{k-1}}\cdots\int_0^{s_5}\frac{s_4^3}{2\cdot
3 }d_{s_4}\cdots
ds_{k-2}ds_{k-1}ds_k\\
=&\cdots=\mathcal{O}\Big(\frac{c^k}{k!}\Big ).
\end{split}
\end{equation*}
We also used the fact that for the first step ($k:=0$), we have
easily
\begin{equation}\label{n0}
\sup_{t\in[0,T]}\mathbf{E} (\| u_{n,1}(\cdot,t)-u_{n,0}(\cdot,t)
\|^p_{L^\infty(\mathcal{D})})\leq c\sup_{t\in[0,T]}\mathbf{E} (\|
u_{n,1}(\cdot,t)\|^p_{L^\infty(\mathcal{D})})+c\sup_{t\in[0,T]}\mathbf{E}
(\|u_{n,0}(\cdot,t) \|^p_{L^\infty(\mathcal{D})})<\infty,
\end{equation}
 since
$u_{n,0}$ is deterministic and $u_{n,1}$ is given by the Picard
scheme involving $f_n(u_{n,0})$ and $\sigma(u_{n,0})$ at the
right-hand side, for $f_n$ and $\sigma$ bounded since Lipschiz. In
details, by Picard scheme, we have
\begin{equation*}
\begin{split}
|u_{n,1}(x,t)|
\le &\int_{0}^{t}|G(x,y,t)|   |u_{0} (y)| dy+\int_{0}^{t}\int_{\mathcal{D}}|\Delta G(x,y,t-s)|
  | f_n(u_{n,0} (y,s))| dyds \\
&+\int_{0}^{t}\int_{\mathcal{D}}|G(x,y,t-s)|   | f_n(u_{n,0}(y,s))| dyds    \\
 &+ \Big |\int_{0}^{t}\int_{\mathcal{D}} G(x,y,t-s)(\sigma(u_{n,0}(y,s))W(dy,ds) \Big
 |.
\end{split}
\end{equation*}
Thus, taking $p$ powers then supremum on $x\in\mathcal{D}$ and
then expectation, exactly as before, using the Green's function
estimates, Burkholder-Davis-Gunty inequality and then H\"older's
inequality, we arrive at
\begin{equation*}
\begin{split}
\sup_{t\in[0,T]}\mathbf{E}(\|u_{n,1}(\cdot,t)\|_{L^\infty(\mathcal{D})}^p)
\le
&\sup_{t\in[0,T]}\mathbf{E}\Big(\Big{\|}\int_{0}^{t}|G(\cdot,y,t)|
|u_{0} (y)| dy\Big{\|}^p_{L^\infty(\mathcal{D})}\Big
)\\
&+c\mathbf{E}\bigg (\int_{0}^{t}1ds\bigg )\le c+c<\infty.
\end{split}
\end{equation*}
So, \eqref{n0} is valid and indeed \eqref{da77} holds true.

Taking now supremum in $t$ at \eqref{da77} we obtain
$$\sup_{t\in[0,T]}\mathbf{E}
(\|u_{n,k+1}(\cdot,t)-u_{n,k}(\cdot,t)\|_{L^\infty(\mathcal{D})}^p)\le
c\frac{c^{2k}}{k!},$$ and by summation, we get
\begin{equation}\label{da78}
\begin{split}
\sum_{k=0}^\infty\sup_{t\in[0,T]}\mathbf{E}
(\|u_{n,k+1}(\cdot,t)-u_{n,k}(\cdot,t)\|_{L^\infty(\mathcal{D})}^p)
\le c\sum_{k=0}^\infty\frac{c^{2k}}{k!}=c\exp(c^2)<\infty,
\end{split}
\end{equation}
for any $p>\max\{2,\frac{8}{4-d}\}=\frac{8}{4-d}$.

Therefore, it follows that, for $n$ fixed, the limit
$\displaystyle{\lim_{k\rightarrow\infty}}u_{n,k}$, in the
$L^p(\Omega)$ norm, exists for any
$(x,t)\in\mathcal{D}\times[0,T]$. Indeed, we have for any
$(x,t)\in\mathcal{D}\times[0,T]$
\begin{equation}\label{eqc}
\begin{split}\mathbf{E} (|u_{n,k+1}(x,t)-u_{n,k}(x,t)|^p)\leq &\mathbf{E}
(\|u_{n,k+1}(\cdot,t)-u_{n,k}(\cdot,t)\|_{L^\infty(\mathcal{D})}^p)\\
\leq& \sup_{t\in[0,T]}\mathbf{E}
(\|u_{n,k+1}(\cdot,t)-u_{n,k}(\cdot,t)\|_{L^\infty(\mathcal{D})}^p)\rightarrow
0\;\;{\rm as}\;\;k\rightarrow\infty.
\end{split}
\end{equation}
So, for $n$ fixed, the sequence $u_{n,k}$ is Cauchy in
$L^p(\Omega)$, and convergent as $k\rightarrow\infty$ to some
$u_n$ in this norm, i.e.,
$$\exists\;u_n:\;\displaystyle{\lim_{k\rightarrow\infty}}
\mathbf{E}\Big{(}|u_{n,k}(\cdot,t)-u_n(\cdot,t)|^p\Big{)}=0.$$

Moreover, we observe that $u_{n,k}$, for $n$ fixed, is also Cauchy
in the norm $L^p(\infty,\Omega)$ defined by
$$\|v(\cdot,t)\|_{L^p(\infty,\Omega)}:=\Big{(}\mathbf{E}
(\|v(\cdot,t)\|_{L^\infty(\mathcal{D})}^p)\Big{)}^{1/p},$$ and
convergent in this norm, i.e.,
$$\exists\;\tilde{u_n}:\;\displaystyle{\lim_{k\rightarrow\infty}}
\mathbf{E}\Big{(}\|u_{n,k}(\cdot,t)-\tilde{u_n}(\cdot,t)\|_{L^\infty(\mathcal{D})}^p\Big{)}=0.$$
Obviously, since
$$\|u_{n,k}(\cdot,t)-\tilde{u_n}(\cdot,t)\|_{L^p(\Omega)}\leq
\|u_{n,k}(\cdot,t)-\tilde{u_n}(\cdot,t)\|_{L^p(\infty,\Omega)},$$
from uniqueness of limits, we have $u_n=\tilde{u_n}$, and thus
$$\mathbf{E}
(\|u_n(\cdot,t)-u_{n,k}(\cdot,t)\|_{L^\infty(\mathcal{D})}^p)\rightarrow
0,\;\;{\rm as}\;k\rightarrow\infty,$$ and so
\begin{equation}\label{eq23}
\mathbf{E}
(\|u_n(\cdot,t)-u_{n,k}(\cdot,t)\|_{L^\infty(\mathcal{D})}^p)<\infty,
\end{equation}
for any $k$.

We then have, using \eqref{da78} and \eqref{eq23}, for any $t$
\begin{equation*}
\begin{split}
\mathbf{E} (\|u_n(\cdot,t)\|_{L^\infty(\mathcal{D})}^p)\leq &
\mathbf{E}
(\|u_n(\cdot,t)-u_{n,k}(\cdot,t)\|_{L^\infty(\mathcal{D})}^p)+c\sum_{j=k}^\infty\mathbf{E}
(\|u_{n,j+1}(\cdot,t)-u_{n,j}(\cdot,t)\|_{L^\infty(\mathcal{D})}^p)\\
\leq& \mathbf{E}
(\|u_n(\cdot,t)-u_{n,k}(\cdot,t)\|_{L^\infty(\mathcal{D})}^p)+c\sum_{j=k}^\infty\sup_{t\in[0,T]}\mathbf{E}
(\|u_{n,j+1}(\cdot,t)-u_{n,j}(\cdot,t)\|_{L^\infty(\mathcal{D})}^p)\\
\leq &\mathbf{E}
(\|u_n(\cdot,t)-u_{n,k}(\cdot,t)\|_{L^\infty(\mathcal{D})}^p)+c\leq
c.
\end{split}
\end{equation*}
Hence, by taking supremum over all $t\in[0,T]$, we obtain for any
$p>\max\{2,8/(4-d)\}=\frac{8}{3}$, in dimensions $d=1$,
\begin{equation}\label{supre}
\sup_{t\in[0,T]}\mathbf{E}
(\|u_n(\cdot,t)\|_{L^\infty(\mathcal{D})}^p)<\infty.
\end{equation}
Note that for power $\hat{p}$ such that $2\leq \hat{p}\leq
\frac{8}{3}$, we use H\"older's inequality for the expectation as
follows. Observe that $2\hat{p}>\frac{8}{3}$, and take
\begin{equation*}
\mathbf{E}
(\|u_n(\cdot,t)\|_{L^\infty(\mathcal{D})}^{\hat{p}})\leq c
\mathbf{E}
(\|u_n(\cdot,t)\|_{L^\infty(\mathcal{D})}^{2\hat{p}})^{1/2}
\end{equation*}
Thus, we get
\begin{equation*}
\sup_{t\in[0,T]}\mathbf{E}
(\|u_n(\cdot,t)\|_{L^\infty(\mathcal{D})}^{\hat{p}})\leq c
\sup_{t\in[0,T]}\mathbf{E}
(\|u_n(\cdot,t)\|_{L^\infty(\mathcal{D})}^{2\hat{p}})^{1/2}<\infty,
\end{equation*}
by \eqref{supre}, since $2\hat{p}>\frac{8}{3}$. So, we have
finally for any $p\geq 2$, in dimensions $d=1$
\begin{equation}\label{supref}
\sup_{t\in[0,T]}\mathbf{E}
(\|u_n(\cdot,t)\|_{L^\infty(\mathcal{D})}^p)<\infty.
\end{equation}

 Through the scheme \eqref{Picard iteration k+1}, by a
standard argument, where we take limits in the $L^p(\Omega)$
norm, and use the fact that $f_n$ and $\sigma$ are uniformly
continuous since Lipschitz, we have
\begin{align} \label{picl}
\lim_{k\rightarrow\infty}u_{n,k}(x,t) =& \int_{\mathcal{D}}u_0(y)G(x,y,t) dy &&\nonumber \\
&+ \int_{0}^{t}\int_{\mathcal{D}}[\Delta G(x,y,t-s)-G(x,y,t-s)] f_n(\lim_{k\rightarrow\infty}u_{n,k}
(y,s)) dy ds &&\nonumber \\
&+
\int_{0}^{t}\int_{\mathcal{D}}G(x,y,t-s)\sigma(\lim_{k\rightarrow\infty}u_{n,k}(y,s))W(dy,ds).
\end{align}
Note that for the stochastic term, since
$$\|u_n(\cdot,t)-u_{n,k}(\cdot,t)\|_{L^p(\Omega)}\rightarrow
0,\;\;{\rm as}\;k\rightarrow\infty,$$ we can easily prove that
$$\Big{\|}\int_{0}^{t}\int_{\mathcal{D}}G(x,y,t-s)[\sigma(u_n(y,s))-\sigma(u_{n,k}(y,s))]W(dy,ds)
\Big{\|}_{L^p(\Omega)}\rightarrow 0,\;\;{\rm
as}\;k\rightarrow\infty,$$ by using Burkholder-Davis-Gundy
inequality as before, the Lipschitz property (or uniform
continuity of $\sigma$), H\"older inequality and the estimates of
$G$.

So, since $u_n=\displaystyle{\lim_{k\rightarrow\infty}}u_{n,k}$,
in the $L^p(\Omega)$ norm, we derive that $u_n$ satisfies the
stochastic pde \eqref{piecewise ff}; as we shall prove in the
sequel, \eqref{piecewise ff} is uniquely solvable (due to the
fact that $f_n$, $\sigma$ are Lipschitz in $\mathbb{R}$).
Moreover, $u=u_n$ on $\Omega_n$ a.s. (see also in \cite{CW1}, for
the analogous argument for the stochastic Cahn-Hilliard case,
where the same cut-off function was used).


We proceed by establishing uniqueness of solution for the problem
$(\ref{piecewise ff})$.

Let us suppose that $\omega_n$ is a solution of $(\ref{piecewise
ff})$. Then since $u_n$ is a solution also, by using
(\ref{piecewise ff}) for $\omega_n$ and $u_n$ respectively, and
subtracting, we get
\begin{equation*}
\begin{split}
u_n(x,t)-\omega_n(x,t)
= &\int_{0}^{t}\int_{\mathcal{D}}[\Delta G(x,y,t-s)-G(x,y,t-s)] \Big (f_n(u_n(y,s)) -f_n(\omega_n(y,s)) \Big ) dy ds \\
&+ \int_{0}^{t}\int_{\mathcal{D}}G(x,y,t-s) \Big
(\sigma(u_n(y,s))  -\sigma(\omega_n(y,s)) \Big )W(dy,ds).
\end{split}
\end{equation*}
Hence, we obtain
\begin{equation*}
\begin{split}
|u_n(x,t)-\omega_n(x,t)|
\le  &\int_{0}^{t}\int_{\mathcal{D}}|\Delta G(x,y,t-s)|   |f_n(u_n(y,s))-f_n(u_n(y,s)) | dyds   \\
&+\int_{0}^{t}\int_{\mathcal{D}}|G(x,y,t-s)|   | f_n(u_n(y,s))-f_n(\omega_n(y,s)) | dyds  \\
 &+ \Big |\int_{0}^{t}\int_{\mathcal{D}} G(x,y,t-s) (\sigma(u_n(y,s))-  \sigma(\omega_n(y,s)) )W(dy,ds) \Big
 |.
\end{split}
\end{equation*}
We take $p$ powers for $p\geq 2$, and proceed as we did for
deriving \eqref{d1n}, i.e., we take supremum in space, expectations
at both sides, use that $f_n$ and $\sigma$ are Lipschitz, and
apply the Burkholder-Davis-Gundy inequality to the stochastic
term. This yields
\begin{equation}\label{d1nun}
\begin{split}
\mathbf{E}
\Big{(}\|u_{n}(\cdot,t)-&\omega_{n}(\cdot,t)\|_{L^\infty(\mathcal{D})}^p\Big{)}\\
\le &c\mathbf{E} \bigg (\sup_{x\in\mathcal{D}} \Big (\int_{0}^{t}
\int_{\mathcal{D}}|\Delta G(x,y,t-s)|   | u_{n}(y,s)-\omega_{n}(y,s) | dyds \Big )^p \bigg )  \\
&+ c\mathbf{E} \bigg ( \sup_{x\in\mathcal{D}}\Big (\int_{0}^{t}
\int_{\mathcal{D}}|G(x,y,t-s)|   | u_{n}(y,s)-\omega_{n}(y,s) | dyds \Big )^p  \bigg ) \\
& +c\mathbf{E} \bigg
(\sup_{\tau\in[0,t]}\sup_{x\in\mathcal{D}}\Big
(\int_{0}^{\tau}\int_{\mathcal{D}}| G(x,y,\tau-s)|^2 |u_{n}(y,s)-
\omega_{n}(y,s) |^2 dyds \Big )^{p/2} \bigg ).
\end{split}
\end{equation}
Observe that the previous inequality is the same as \eqref{d1n},
where the differences $u_{n,k+1}-u_{n,k}$, $u_{n,k}-u_{n,k-1}$ are
replaced by $u_n-\omega_n$. Thus, a direct result is the
analogous of \eqref{d55}, i.e., for any
$p>\max\{2,8/(4-d)\}=\frac{8}{3}$
\begin{equation}\label{d55un}
\begin{split}
\mathbf{E}
(\|u_{n}(\cdot,t)-\omega_{n}(\cdot,t)\|_{L^\infty(\mathcal{D})}^p)
&\le c\mathbf{E} \bigg (\int_{0}^{t} \|
u_{n}(\cdot,s)-\omega_{n}(\cdot,s)
\|^p_{L^\infty(\mathcal{D})}ds \bigg )\\
&\le c\bigg (\int_{0}^{t}\mathbf{E}( \|
u_{n}(\cdot,s)-\omega_{n}(\cdot,s) \|^p_{L^\infty(\mathcal{D})})ds
\bigg ),
\end{split}
\end{equation}
where again we used Fubini's Theorem.

Hence, by applying Gronwall's Lemma to the previous inequality for
the term $\mathbf{E}( \| u_{n}(\cdot,t)-\omega_{n}(\cdot,t)
\|^p_{L^\infty(\mathcal{D})})$, we obtain
$$\mathbf{E} (\|
u_{n}(\cdot,t)-\omega_{n}(\cdot,t)
\|^p_{L^\infty(\mathcal{D})})\leq 0,$$ for any $t\in[0,T]$. So
for any $t$ in $[0,T]$,
$$\mathbf{E} (\|
u_{n}(\cdot,t)-\omega_{n}(\cdot,t) \|^p_{L^\infty(\mathcal{D})})=
0.$$ This yields that $u_n(x,t)=\omega_n(x,t)$ almost surely in
$\Omega$ and in $\Omega_n$ (since $\Omega_n\subset\Omega$ and
thus $\|v\|_{L^p(\Omega_n)}\leq \|v\|_{L^p(\Omega)}$), for any
$t$, $x$, i.e., for $\hat{\Omega}:=\Omega,\;{\rm or}\;\Omega_n$
$$P\Big{(}w\in\hat{\Omega}:\;u_n(x,t;w)=
\omega_n(x,t;w)\Big{)}=1,\;\;{\rm for\;any}\;\;t\in[0,T],\;\;{\rm
and\; any}\;\;x\in\mathcal{D},$$
 and so by definition $u_n$, $\omega_n$ are
equivalent in $\Omega$ and in $\Omega_n$.

We shall use now the fact that when two processes are equivalent
in a set and a.s. continuous in the same set, then they are
indistinguishable in this set.

The solution $u$ of the stochastic Cahn-Hilliard/Allen-Cahn
equation \eqref{sm} is almost surely continuous in space and
time, in dimensions $d=1$, cf. \cite{AKM}, and the approximations
$u_n,\;\omega_n$ of $u$ satisfy the equation \eqref{sm} a.s. in
$\Omega_n$ (since $f_n(u_n)=f(u_n)$ and
$f_n(\omega_n)=f(\omega_n)$ in $\Omega_n$ a.s.). So, the
equivalent processes $u_n,\;\omega_n$ are almost surely
continuous in $\Omega_n$ also and thus indistinguishable in
$\Omega_n$ (having the same paths), i.e.,
\begin{equation}\label{i1}
P\Big{(}w\in\Omega_n:\;u_n(x,t;w)=\omega_n(x,t;w),\;\mbox{for any
}(x,t)\in\mathcal{D}\times[0,T]\Big{)}=1.
\end{equation}
Since $u_n$, $\omega_n$ are indistinguishable on $\Omega_n$ then
we have uniqueness of solution of (\ref{piecewise ff}) with
uniquely defined paths a.s. on $\Omega_n$.

Thus, $u_n$ is well defined by (\ref{piecewise ff}), and suitable
for localizing $u$.
\end{proof}

\subsubsection{The Malliavin derivative of $u_n$}

We proceed by proving that the derivative of the approximation
$u_n$ in the Malliavin sense, is well defined as the solution of
an spde. In addition, we establish the regularity of $u_n$ in
$D^{1,2}$ and $L^{1,2}$; this is accomplished at the next
proposition.

\begin{prop} \label{LL234}
Let $u_n(x,t)$ be the solution of $(\ref{piecewise ff})$, then:
\begin{enumerate}
\item
$u_n$ belongs to the space $D^{1,2}$.
\item
The Malliavin derivative of $u_n$ satisfies for any $s\leq t$,
uniquely, the spde of the form
\begin{equation} \label{MDE}
\begin{split}
 D_{y,s}u_{n}(x,t):=D_{y,s}(u_{n}(x,t))=&\int_{s}^{t}\int_{\mathcal{D}}[\Delta
G(x,z,t-\tau)-G(x,z,t-\tau)]
\tilde{\mathcal{G}_2}(n)(z,\tau)D_{y,s}(u_{n}(z,\tau)) dz d\tau\\
&+G(x,y,t-s)\sigma(u_{n}(y,s))\\
&+\int_{s}^{t}\int_{\mathcal{D}}G(x,z,t-\tau)\tilde{\mathcal{G}_1}(n)(z,\tau)D_{y,s}(u_{n}(z,\tau))W(dz,d\tau),
\end{split}
\end{equation}
while
$$D_{y,s}u_n(x,t)=0\;\;\mbox{ for any }s>t.$$

Here, $\tilde{\mathcal{G}_1}(n)(z,\tau)$,
$\tilde{\mathcal{G}_2}(n)(z,\tau)$ are bounded, and satisfy
$$D_{y,s} (\sigma(u_n(z,\tau))) = \tilde{\mathcal{G}_1}(n)(z,\tau)D_{y,s}(u_n(z,\tau)),$$
$$D_{y,s} (f_n(u_n(z,\tau))) = \tilde{\mathcal{G}_2}(n)(z,\tau)D_{y,s}(u_n(z,\tau)).$$
\item
$u_n$ belongs to $L^{1,2}$.
\end{enumerate}
\end{prop}
\begin{proof}
First, we will prove that the Cauchy sequence $\{u_{n,k} \}_{k
\in \mathbb{N}}$ (as we described in Lemma (\ref{lemma 2.1}))
belongs  to the space $D^{1,2}$ for all $(x,t) \in [0,T] \times
\mathcal{D}$, by using induction and the Picard iteration scheme.


For $k=0$, the function $u_{n,0}$ is deterministic with Malliavin
derivative $Du_{n,0}=0$. Thus $u_{n,0} \in D^{1,2}.$

We proceed with induction.

We suppose for $k \ge 0 $ that for any $i\leq k$, $u_{n,i} \in
D^{1,2}$ for every $(x,t) \in [0,T] \times \mathcal{D}$, and that
$$
\sup_{t\in[0,T]}\sup_{i\leq
k}\mathbf{E}\Big{(}\int_0^t\int_{\mathcal{D}}\|
D_{s,y}u_{n,i}(\cdot,t)\|_{L^\infty(\mathcal{D})}^2dyds\Big{)}<\infty.
$$
We shall prove that for any $i\leq k+1$, $u_{n,i} \in D^{1,2}$
for every $(x,t) \in [0,T] \times \mathcal{D}$ also (i.e.,
$u_{n,k+1} \in D^{1,2}$ for every $(x,t) \in [0,T] \times
\mathcal{D}$), and
$$
\sup_{t\in[0,T]}\sup_{i\leq
k+1}\mathbf{E}\Big{(}\int_0^t\int_{\mathcal{D}}\|
D_{s,y}u_{n,i}(\cdot,t)\|_{L^\infty(\mathcal{D})}^2dyds\Big{)}<\infty,
$$
also (the bounds being independent of $k$). Note that the integral
for $s\in[0,t]$ coincides with the integral for $s\in[0,T]$, since
the Malliavin derivative involved is zero for any $s>t$. This
will result that
\begin{equation}\label{indmal}
\begin{split}
&\forall\;k\;\;\;\exists\;\;u_{n,k} \in D^{1,2}\;\;\forall\;(x,t)
\in \mathcal{D}\times[0,T],\;\;\mbox{and}\\
&\sup_{t\in[0,T]}\sup_{k}\mathbf{E}\Big{(}\int_0^T\int_{\mathcal{D}}\|
D_{s,y}u_{n,k}(\cdot,t)\|_{L^\infty(\mathcal{D})}^2dyds\Big{)}<\infty.
\end{split}
\end{equation}


We apply the Malliavin derivative to $(\ref{Picard iteration
k+1})$, and get, since it is a linear operator
\begin{equation}\label{mdk}
\begin{split}
D_{y,s}(u_{n,k+1}(x,t))=&:D_{y,s}u_{n,k+1}(x,t) = D_{y,s}\Big{[}\int_{\mathcal{D}}u_0(y)G(x,z,t) dz\Big{]} \\
&+ D_{y,s}\Big{[}\int_{0}^{t}\int_{\mathcal{D}}[\Delta G(x,z,t-\tau)-G(x,z,t-\tau)] f_n(u_{n,k}(z,\tau))
 dz d\tau\Big{]}\\
&+
D_{y,s}\Big{[}\int_{0}^{t}\int_{\mathcal{D}}G(x,z,t-\tau)\sigma(u_{n,k}(z,\tau))W(dz,d\tau)\Big{]}\\
=&0+\int_{0}^{t}\int_{\mathcal{D}}D_{y,s}\Big{(}[\Delta G(x,z,t-\tau)-G(x,z,t-\tau)] f_n(u_{n,k}(z,\tau))
\Big{)} dz d\tau\\
&+G(x,y,t-s)\sigma(u_{n,k}(y,s))\\
&+\int_{0}^{t}\int_{\mathcal{D}}D_{y,s}\Big{(}G(x,z,t-\tau)\sigma(u_{n,k}(z,\tau))\Big{)}W(dz,d\tau)\\
=&\int_{0}^{t}\int_{\mathcal{D}}D_{y,s}\Big{(}\Delta
G(x,z,t-\tau)-G(x,z,t-\tau)\Big{)}
f_n(u_{n,k}(z,\tau)) dz d\tau\\
&+\int_{0}^{t}\int_{\mathcal{D}}[\Delta
G(x,z,t-\tau)-G(x,z,t-\tau)]
D_{y,s}\Big{(}f_n(u_{n,k}(z,\tau))\Big{)} dz d\tau\\
&+G(x,y,t-s)\sigma(u_{n,k}(y,s))\\
&+\int_{0}^{t}\int_{\mathcal{D}}D_{y,s}\Big{(}G(x,z,t-\tau)\Big{)}\sigma(u_{n,k}(z,\tau))W(dz,d\tau)\\
&+\int_{0}^{t}\int_{\mathcal{D}}G(x,z,t-\tau)D_{y,s}\Big{(}\sigma(u_{n,k}(z,\tau))\Big{)}W(dz,d\tau)\\
=&0+\int_{0}^{t}\int_{\mathcal{D}}[\Delta
G(x,z,t-\tau)-G(x,z,t-\tau)]
D_{y,s}\Big{(}f_n(u_{n,k}(z,\tau))\Big{)} dz d\tau\\
&+G(x,y,t-s)\sigma(u_{n,k}(y,s))\\
&+0+\int_{0}^{t}\int_{\mathcal{D}}G(x,z,t-\tau)D_{y,s}\Big{(}\sigma(u_{n,k}(z,\tau))\Big{)}W(dz,d\tau)\\
=&\int_{s}^{t}\int_{\mathcal{D}}[\Delta
G(x,z,t-\tau)-G(x,z,t-\tau)]
D_{y,s}\Big{(}f_n(u_{n,k}(z,\tau))\Big{)} dz d\tau\\
&+G(x,y,t-s)\sigma(u_{n,k}(y,s))\\
&+\int_{s}^{t}\int_{\mathcal{D}}G(x,z,t-\tau)D_{y,s}\Big{(}\sigma(u_{n,k}(z,\tau))\Big{)}W(dz,d\tau),
\end{split}
\end{equation}
where we used also that the Malliavin derivative is zero when
applied to the deterministic terms $G$, $\Delta G$ (since no
change is observed on $\omega\in\Omega$, they are constant as
functions of $\omega\in\Omega$). Moreover, since the Malliavin
derivative is zero for any $\tau<s$, this resulted to integrals
on $\tau\geq s$.

Here, we note that $D_{y,s}(u_{n,k+1}(x,t))$ is a function of
$y,\;s,\;x,\;t$. In this work, the notation $D_{y,s}f(x,t)$, for
a general function $f$, is used to denote $D_{y,s}(f(x,t))$.

We now use Proposition 1.2.4 of \cite{N}, cf. also in \cite{CW1},
in dimensions $m=1$ (following the Nualart's book notation, since
$u_{n,k}(x,t)\in\mathbb{R}^m$, $m=1$) with the norm used for the
Lipschitz condition being the absolute value. More specifically,
since $u_{n,k}$ belongs to $D^{1,2}$ (true by the induction
hypothesis) and $\sigma$ is Lipschitz uniformly on any $x$ in
$\mathbb{R}$ with $K_\sigma$ its Lipschitz coefficient, then
$\sigma(u_{n,k})$ belongs to $D^{1,2}$ also, and there exists a
random variable $\mathcal{G}_1=\mathcal{G}_1(n,k)$ such that
\begin{equation}\label{form1}
D_{y,s}\Big{(}\sigma(u_{n,k}(x,t))\Big{)}=\mathcal{G}_1(n,k)(x,t)D_{y,s}u_{n,k}(x,t),
\end{equation}
with $\mathcal{G}_1$ bounded (in the absolute value norm) by
$K_\sigma$, uniformly for any $x$, $t$, i.e.,
$$|\mathcal{G}_1(n,k)(x,t)|\leq
K_\sigma,\;\;\forall\;x\in\mathcal{D},\;\forall t\in[0,T].$$
Since $K_\sigma$ is independent of $n$, $k$, we have finally
\begin{equation}\label{lipder}
\sup_{n,k,(x,t)\in
\mathcal{D}\times[0,T]}|\mathcal{G}_1(n,k)(x,t)|\leq K_\sigma.
\end{equation}

The same argument can be applied for $f_n$ in place of $\sigma$,
since $f_n$ is also Lipschitz uniformly on $\mathbb{R}$. Indeed,
there exists a random variable $\mathcal{G}_2=\mathcal{G}_2(n,k)$
such that
\begin{equation}\label{form2}
D_{y,s}\Big{(}f_n(u_{n,k}(x,t))\Big{)}=\mathcal{G}_2(n,k)(x,t)D_{y,s}u_{n,k}(x,t),
\end{equation}
and
\begin{equation}\label{lipder2}
\sup_{k,(x,t)\in
\mathcal{D}\times[0,T]}|\mathcal{G}_2(n,k)(x,t)|\leq K_{f_n},
\end{equation}
for $K_{f_n}$ a positive constant, depending on $n$ through $f_n$.

Therefore, \eqref{form1} and \eqref{form2}, together with
\eqref{mdk}, give finally for any $s\leq t$
\begin{equation}\label{mdkn}
\begin{split}
D_{y,s}u_{n,k+1}(x,t)=&\int_{s}^{t}\int_{\mathcal{D}}[\Delta
G(x,z,t-\tau)-G(x,z,t-\tau)]
\mathcal{G}_2(n,k)(z,\tau)D_{y,s}u_{n,k}(z,\tau) dz d\tau\\
&+G(x,y,t-s)\sigma(u_{n,k}(y,s))\\
&+\int_{s}^{t}\int_{\mathcal{D}}G(x,z,t-\tau)\mathcal{G}_1(n,k)(z,\tau)D_{y,s}u_{n,k}(z,\tau)W(dz,d\tau),
\end{split}
\end{equation}
while for $s>t$
$$D_{y,s}u_{n,k+1}(x,t)=0.$$
Taking absolute value at both sides of \eqref{mdkn}, and then $p$
powers for $p\geq 2$ , we get
\begin{equation*}
\begin{split}
 | D_{s,y}u_{n,k+1}(x,t)|^p \le &c|G(x,y,t-s)\sigma(u_{n,k}(y,s)) |^p \\
&+c\Big | \int_{s}^{t}\int_{\mathcal{D}}[\Delta G(x,z,
t-\tau)-G(x,z, t-\tau)] \mathcal{G}_2(n,k)(z,\tau)D_{s,y}u_{n,k}(z, \tau)  dz d \tau \Big |^p   \\
&+c \Big |\int_{s}^{t}\int_{\mathcal{D}}G(x,z, t-\tau)
\mathcal{G}_1(n,k)(z,\tau)D_{s,y} u_{n,k}(z,\tau) W(dz,d \tau)
\Big |^p,
\end{split}
\end{equation*}
which gives by \eqref{lipder2}
\begin{equation*}
\begin{split}
 \| D_{s,y}u_{n,k+1}(\cdot,t)\|_{L^\infty(\mathcal{D})}^p \le
 &c\|G(\cdot,y,t-s)\sigma(u_{n,k}(y,s)) \|_{L^\infty(\mathcal{D})}^p \\
&+cK_{f_n}\Big \| \int_{s}^{t}\int_{\mathcal{D}}|\Delta G(\cdot,z,
t-\tau)-G(\cdot,z, t-\tau)|
|D_{s,y}u_{n,k}(z, \tau)|  dz d \tau \Big \|_{L^\infty(\mathcal{D})}^p   \\
&+c \Big \|\int_{s}^{t}\int_{\mathcal{D}}G(\cdot,z, t-\tau)
\mathcal{G}_1(n,k)(z,\tau)D_{s,y} u_{n,k}(z,\tau) W(dz,d \tau)
\Big \|_{L^\infty(\mathcal{D})}^p.
\end{split}
\end{equation*}
We integrate for $y\in\mathcal{D},\;s\in [0,t]$ and then take
expectation, to derive
\begin{equation}\label{mm}
\begin{split}
 \mathbf{E}\Big{(}&\int_0^t\int_{\mathcal{D}}\| D_{s,y}u_{n,k+1}(\cdot,t)\|_{L^\infty(\mathcal{D})}^pdyds\Big{)}
  \le c\mathbf{E}\Big{(}\int_0^t\int_{\mathcal{D}}\|G(\cdot,y,t-s)
  \sigma(u_{n,k}(y,s)) \|_{L^\infty(\mathcal{D})}^pdyds\Big{)} \\
&+cK_{f_n}\mathbf{E}\Big{(}\int_0^t\int_{\mathcal{D}}\Big \|
\int_{s}^{t}\int_{\mathcal{D}}|\Delta G(\cdot,z,
t-\tau)-G(\cdot,z, t-\tau)| |D_{s,y}u_{n,k}(z, \tau)|  dz d \tau \Big \|_{L^\infty(\mathcal{D})}^pdyds\Big{)}   \\
&+c \mathbf{E}\Big{(}\int_0^t\int_{\mathcal{D}}\Big
\|\int_{s}^{t}\int_{\mathcal{D}}G(\cdot,z, t-\tau)
\mathcal{G}_1(n,k)(z,\tau)D_{s,y} u_{n,k}(z,\tau) W(dz,d
\tau)\Big \|_{L^\infty(\mathcal{D})}^pdyds\Big{)}\\
:=&M_1(t;k)+M_2(t;k)+M_3(t;k).
\end{split}
\end{equation}

We shall estimate the terms $M_{i}(t;k)$ for $i=1,2,3$.

Considering the term $M_{1}(t;k)$, we have
\begin{equation}\label{d78}
\begin{split}
M_{1}(t;k)=&c\mathbf{E}\Big{(}\int_0^t\int_{\mathcal{D}}\|G(\cdot,y,t-s)\sigma(u_{n,k}(y,s))
\|_{L^\infty(\mathcal{D})}^pdyds\Big{)}\\
\leq& c
\mathbf{E}\Big{(}\int_0^t\int_{\mathcal{D}}\|G(\cdot,y,t-s)\|_{L^\infty(\mathcal{D})}^{2p}dyds\Big{)}
+c\mathbf{E}\Big{(}\int_0^t\int_{\mathcal{D}}|\sigma(u_{n,k}(y,s))
|^{2p}dyds\Big{)}\\
\leq& c
\mathbf{E}\Big{(}\int_0^t\int_{\mathcal{D}}\|G(\cdot,y,t-s)\|_{L^\infty(\mathcal{D})}^{2p}dyds\Big{)}
+c\mathbf{E}\Big{(}\int_0^t\int_{\mathcal{D}}c(1+|u_{n,k}(y,s)|^{2pq})
dyds\Big{)}\\
\leq& c
\mathbf{E}\Big{(}\int_0^t\int_{\mathcal{D}}\|G(\cdot,y,t-s)\|_{L^\infty(\mathcal{D})}^{2p}dyds\Big{)}
+c+c\mathbf{E}\Big{(}\int_0^tc\|u_{n,k}(\cdot,s)\|_{L^\infty(\mathcal{D})}^{2pq}ds\Big{)}\\
\leq& c+c
\mathbf{E}\Big{(}\int_0^t\int_{\mathcal{D}}\|G(\cdot,y,t-s)\|_{L^\infty(\mathcal{D})}^{2p}dyds\Big{)}
+c\int_0^t\mathbf{E}\Big{(}\|u_{n,k}(\cdot,s)\|_{L^\infty(\mathcal{D})}^{2pq}\Big{)}ds,
\end{split}
\end{equation}
where we used the growth of the unbounded noise diffusion, for
$q\in(0,1/3)$, and Fubini's Theorem.

We use the next estimate (1.6) of \cite{CW1}, to get
\begin{equation}\label{m11}
\int_0^{t}\int_{\mathcal{D}}\|G(\cdot,y,t-s)\|_{L^\infty(\mathcal{D})}^{2p}dy
ds \le C \int_0^t|t-s|^{-2pd/4+d/4}ds<\infty,
\end{equation}
for $-2pd/4+d/4=(-2p+1)/4>-1$ (since $d=1$) i.e., for $(2\leq
)p<5/2$.

Also since $2pq\leq 2p<5$ in dimensions $d=1$, using \eqref{supre}
and \eqref{eq23}, we obtain for any $t\in[0,T]$
\begin{equation}\label{m12}
\begin{split}
\mathbf{E}(\|u_{n,k}(\cdot,t)\|_{L^\infty(\mathcal{D})}^{2pq})\leq&
c+c\mathbf{E}(\|u_{n,k}(\cdot,t)-u_n(\cdot,t)\|_{L^\infty(\mathcal{D})}^5)
+c\mathbf{E}(\|u_{n}(\cdot,t)\|_{L^\infty(\mathcal{D})}^5)\\
\leq&
c+c\sup_{t\in[0,T]}\mathbf{E}(\|u_{n}(\cdot,t)\|_{L^\infty(\mathcal{D})}^5)\leq
c+c\leq c.
\end{split}
\end{equation}

Using \eqref{m11}, \eqref{m12} in \eqref{d78}, yields for $2\leq
p<5/2$
\begin{equation}\label{d79}
\sup_{k,t\in[0,T]}M_{1}(t;k)<\infty.
\end{equation}

Considering the term $M_{2}(t;k)$, we choose the inequality
(1.12) of \cite{CW1}, p. 781, for $$H(x,y,t-\tau):=|\Delta
G(x,z,t-\tau)|,\;\;\;\upsilon( z ,\tau):=|D_{s,y}u_{n,k}(z,
\tau)|,\;\;\;\rho=q\mbox{ (of \cite{CW1} notation)
}=\infty,\;\;\;r=1.$$ As in \eqref{da1}, we have
\begin{equation}\label{da1n}
\begin{split}
\Big \| \int_{s}^{t}\int_{\mathcal{D}}|\Delta G(\cdot,z,
t-\tau)||D_{s,y}u_{n,k}(z, \tau)|  dz d \tau \Big
\|_{L^\infty(\mathcal{D})}^p \le c\int_0^t \|
D_{s,y}u_{n,k}(\cdot, \tau) \|^p_{L^\infty(\mathcal{D})} d\tau .
\end{split}
\end{equation}
Using the inequality (1.11) of \cite{CW1}, p. 781, for
$$H(x,z,t-\tau)= |G(x,z,t-\tau)|,\;\;\; \upsilon( z ,\tau):=|D_{s,y}u_{n,k}(z,
\tau)|,\;\;\;\rho=q=\infty,\;\;\;r=1,$$ we get
\begin{equation}\label{da2n}
\begin{split}
\Big \| \int_{s}^{t}\int_{\mathcal{D}}\|G(\cdot,z,
t-\tau)||D_{s,y}u_{n,k}(z, \tau)|  dz d \tau \Big
\|_{L^\infty(\mathcal{D})}^p  \le c\int_0^t \|
D_{s,y}u_{n,k}(\cdot, \tau)\|^p_{L^\infty(\mathcal{D})} d\tau .
\end{split}
\end{equation}
Relations \eqref{da1n}, \eqref{da2n} yield
\begin{equation}\label{he2}
\begin{split}
M_{2}(t;k)=&cK_{f_n}\mathbf{E}\Big{(}\int_0^t\int_{\mathcal{D}}\Big
\| \int_{s}^{t}\int_{\mathcal{D}}|\Delta G(\cdot,z,
t-\tau)-G(\cdot,z, t-\tau)||D_{s,y}u_{n,k}(z, \tau)|  dz d \tau
\Big \|_{L^\infty(\mathcal{D})}^pdyds\Big{)}\\
\leq&cK_{f_n}\mathbf{E}\Big{(}\int_0^t\int_{\mathcal{D}}\int_0^t
\| D_{s,y}u_{n,k}(\cdot, \tau)\|^p_{L^\infty(\mathcal{D})} d\tau
dyds\Big )
\\
\leq&cK_{f_n}\int_0^t\mathbf{E}\Big{(}\int_0^t
\int_{\mathcal{D}}\| D_{s,y}u_{n,k}(\cdot,
\tau)\|^p_{L^\infty(\mathcal{D})} dyds\Big ) d\tau\\
=&cK_{f_n}\int_0^t\mathbf{E}\Big{(}\int_0^{\tau}
\int_{\mathcal{D}}\| D_{s,y}u_{n,k}(\cdot,
\tau)\|^p_{L^\infty(\mathcal{D})} dyds\Big ) d\tau,
\end{split}
\end{equation}
where we used Fubini's Theorem; the integral for $s$ is taken
finally in $[0,\tau]$ since for $s>\tau$ the Malliavin derivative
satisfies $D_{s,y}u_{n,k}(x, \tau)=0$, for any $x$.

For the term $M_3(t;k)$, we have, using Fubini's Theorem and
Burkholder-Davis-Gundy inequality
\begin{equation}\label{he4}
\begin{split}
M_3(k;t)=&c \mathbf{E}\Big{(}\int_0^t\int_{\mathcal{D}}\Big
\|\int_{s}^{t}\int_{\mathcal{D}}G(\cdot,z, t-\tau)
\mathcal{G}_1(n,k)(z,\tau)D_{s,y} u_{n,k}(z,\tau) W(dz,d
\tau)\Big \|_{L^\infty(\mathcal{D})}^pdyds\Big{)}\\
\leq&c \int_0^t\int_{\mathcal{D}}\mathbf{E}\Big{(}\Big
\|\int_{s}^{t}\int_{\mathcal{D}}G(\cdot,z, t-\tau)
\mathcal{G}_1(n,k)(z,\tau)D_{s,y} u_{n,k}(z,\tau) W(dz,d
\tau)\Big \|_{L^\infty(\mathcal{D})}^p\Big{)}dyds\\
\leq&c
\int_0^t\int_{\mathcal{D}}\mathbf{E}\Big{(}\sup_{r\in[0,t]}\sup_{x\in\mathcal{D}}\Big
|\int_{0}^{r}\int_{\mathcal{D}}G(x,z, t-\tau)
\mathcal{G}_1(n,k)(z,\tau)D_{s,y} u_{n,k}(z,\tau) W(dz,d
\tau)\Big |^p\Big{)}dyds\\
\leq&c
\int_0^t\int_{\mathcal{D}}\mathbf{E}\Big{(}\sup_{r\in[0,t]}\sup_{x\in\mathcal{D}}\Big
|\int_{0}^{r}\int_{\mathcal{D}}|G(x,z, t-\tau)|^2
|\mathcal{G}_1(n,k)(z,\tau)|^2|D_{s,y} u_{n,k}(z,\tau)|^2 dzd
\tau\Big |^{p/2}\Big{)}dyds\\
\leq&c
\int_0^t\int_{\mathcal{D}}\mathbf{E}\Big{(}\sup_{r\in[0,t]}\sup_{x\in\mathcal{D}}\Big
(\int_{0}^{r}\int_{\mathcal{D}}|G(x,z, t-\tau)|^2 |D_{s,y}
u_{n,k}(z,\tau)|^2 dzd \tau\Big )^{p/2}\Big{)}dyds\\
=&c
\int_0^t\int_{\mathcal{D}}\mathbf{E}\Big{(}\sup_{r\in[0,t]}\Big
\|\int_{0}^{r}\int_{\mathcal{D}}|G(\cdot,z, t-\tau)|^2 |D_{s,y}
u_{n,k}(z,\tau)|^2 dzd \tau\Big
\|_{L^\infty(\mathcal{D})}^{p/2}\Big{)}dyds,
\end{split}
\end{equation}
where we also used the relation \eqref{lipder}.

As in \eqref{da3}, we choose the inequality (1.13) of \cite{CW1},
p. 781, for
$$H(x,y,t-s)= G^2(x,z,t-\tau),\;\;\; \upsilon(z
,\tau)=|D_{s,y} u_{n,k}(z,\tau)|^2,\;\;\;\rho=q\mbox{ (following
\cite{CW1} notation)}=\infty,\;\;\;r=1,$$ and we obtain
\begin{equation}\label{da3n}
\begin{split}
&\mathbf{E}\Big{(}\sup_{r\in[0,t]}\Big
\|\int_{0}^{r}\int_{\mathcal{D}}|G(\cdot,z, t-\tau)|^2 |D_{s,y}
u_{n,k}(z,\tau)|^2 dzd \tau\Big
\|_{L^\infty(\mathcal{D})}^{p/2}\Big{)}\\
&\le \mathbf{E}\Big{(}\sup_{r\in[0,t]}\Big \|\int_{0}^{r}\|D_{s,y}
u_{n,k}(\cdot,\tau)\|_{L^\infty(\mathcal{D})}^2\int_{\mathcal{D}}|G(\cdot,z,
t-\tau)|^2dzd\tau\Big
\|_{L^\infty(\mathcal{D})}^{p/2}\Big{)}\\
&\le \mathbf{E}\Big{(}\Big (\int_{0}^{t}\|D_{s,y}
u_{n,k}(\cdot,\tau)\|_{L^\infty(\mathcal{D})}^2(t-\tau)^{-2d/4+d/4}d\tau\Big
)^{p/2}\Big{)}\\
&= \mathbf{E}\Big{(}\Big (\int_{0}^{t}\|D_{s,y}
u_{n,k}(\cdot,\tau)\|_{L^\infty(\mathcal{D})}^2(t-\tau)^{-d/4}d\tau\Big
)^{p/2}\Big{)}\\
&= \mathbf{E}\Big{(}\Big (\int_{0}^{t}\|D_{s,y}
u_{n,k}(\cdot,\tau)\|_{L^\infty(\mathcal{D})}^2(t-\tau)^{-d/4}d\tau\Big
)\Big{)},
\end{split}
\end{equation}
where we took $p=2$. We use now estimate \eqref{da3n} to
\eqref{he4}, and arrive at
\begin{equation}\label{he5}
\begin{split}
M_3(k;t)\leq &c
\int_0^t\int_{\mathcal{D}}\mathbf{E}\Big{(}\sup_{r\in[0,t]}\Big
\|\int_{0}^{r}\int_{\mathcal{D}}|G(\cdot,z, t-\tau)|^2 |D_{s,y}
u_{n,k}(z,\tau)|^2 dzd \tau\Big
\|_{L^\infty(\mathcal{D})}\Big{)}dyds\\
\leq &c \int_0^t\int_{\mathcal{D}}\mathbf{E} \bigg
(\int_0^t(t-\tau)^{-d/4}\|D_{s,y}
u_{n,k}(\cdot,\tau)\|_{L^\infty(\mathcal{D})}^2d\tau \bigg
)dyds\\
=&c \mathbf{E} \bigg
(\int_0^t\int_{\mathcal{D}}\int_0^t(t-\tau)^{-d/4}\|D_{s,y}
u_{n,k}(\cdot,\tau)\|_{L^\infty(\mathcal{D})}^2d\tau dyds\bigg )
\\
=&c \mathbf{E} \bigg
(\int_0^t\int_0^t\int_{\mathcal{D}}(t-\tau)^{-d/4}\|D_{s,y}
u_{n,k}(\cdot,\tau)\|_{L^\infty(\mathcal{D})}^2dydsd\tau \bigg )
\\
=&c \mathbf{E} \bigg
(\int_0^t\int_0^\tau\int_{\mathcal{D}}(t-\tau)^{-d/4}\|D_{s,y}
u_{n,k}(\cdot,\tau)\|_{L^\infty(\mathcal{D})}^2dydsd\tau \bigg )\\
=&c \mathbf{E} \bigg
(\int_0^t(t-\tau)^{-d/4}\int_0^\tau\int_{\mathcal{D}}\|D_{s,y}
u_{n,k}(\cdot,\tau)\|_{L^\infty(\mathcal{D})}^2dydsd\tau \bigg
)\\
=&c \int_0^t(t-\tau)^{-d/4}\mathbf{E} \bigg
(\int_0^\tau\int_{\mathcal{D}}\|D_{s,y}
u_{n,k}(\cdot,\tau)\|_{L^\infty(\mathcal{D})}^2dyds\bigg )d\tau,
\end{split}
\end{equation}
since for $s>\tau$, $D_{s,y}u_{n,k}(x, \tau)=0$, for any $x$.

Thus, choosing $p=2$ on \eqref{mm}, and using the estimates
\eqref{d79}, \eqref{he2} and \eqref{he5}, we finally proved since
$d=1$
\begin{equation}\label{d80}
\begin{split}
\mathbf{E}\Big{(}\int_0^t\int_{\mathcal{D}}\|
D_{s,y}u_{n,k+1}(\cdot,t)\|_{L^\infty(\mathcal{D})}^2dyds\Big{)}
\leq& C_0+cK_{f_n}\int_0^t\mathbf{E}\Big{(}\int_0^{\tau}
\int_{\mathcal{D}}\| D_{s,y}u_{n,k}(\cdot,
\tau)\|^2_{L^\infty(\mathcal{D})} dyds\Big ) d\tau\\
&+ c \int_0^t(t-\tau)^{-1/4}\mathbf{E} \bigg
(\int_0^\tau\int_{\mathcal{D}}\|D_{s,y}
u_{n,k}(\cdot,\tau)\|_{L^\infty(\mathcal{D})}^2dyds\bigg )d\tau,
\end{split}
\end{equation}
for $C_0,c>0$ constants independent of $k,\;t$.

We take supremum on $i\leq k$ (the above inequality is true for
any such $i$, from the first induction hypothesis:
$D_{s,y}u_{n,i}\in \mathbf{D}_{1,2}$ for any $i\leq k$), and get
\begin{equation*}
\begin{split}
&\sup_{i\leq k}\mathbf{E}\Big{(}\int_0^t\int_{\mathcal{D}}\|
D_{s,y}u_{n,i+1}(\cdot,t)\|_{L^\infty(\mathcal{D})}^2dyds\Big{)}
\leq C_0+cK_{f_n}\int_0^t\sup_{i\leq
k}\mathbf{E}\Big{(}\int_0^{\tau} \int_{\mathcal{D}}\|
D_{s,y}u_{n,i}(\cdot,
\tau)\|^2_{L^\infty(\mathcal{D})} dyds\Big ) d\tau\\
&+ c \int_0^t(t-\tau)^{-1/4}\sup_{i\leq k}\mathbf{E} \bigg
(\int_0^\tau\int_{\mathcal{D}}\|D_{s,y}
u_{n,i}(\cdot,\tau)\|_{L^\infty(\mathcal{D})}^2dyds\bigg )d\tau\\
\leq& c+C_0+cK_{f_n}\int_0^t\sup_{i\leq
k}\mathbf{E}\Big{(}\int_0^{\tau} \int_{\mathcal{D}}\|
D_{s,y}u_{n,i+1}(\cdot,
\tau)\|^2_{L^\infty(\mathcal{D})} dyds\Big ) d\tau\\
&+ c \int_0^t(t-\tau)^{-1/4}\sup_{i\leq k}\mathbf{E} \bigg
(\int_0^\tau\int_{\mathcal{D}}\|D_{s,y}
u_{n,i+1}(\cdot,\tau)\|_{L^\infty(\mathcal{D})}^2dyds\bigg )d\tau,
\end{split}
\end{equation*}
which gives for
$$A_{n,k+1}(t):=\sup_{i\leq k+1}\mathbf{E}\Big{(}\int_0^t\int_{\mathcal{D}}\|
D_{s,y}u_{n,i}(\cdot,t)\|_{L^\infty(\mathcal{D})}^2dyds\Big{)},$$
\begin{equation}\label{d81}
\begin{split}
A_{n,k+1}(t) \leq& c+C_0+cK_{f_n}\int_0^tA_{n,k+1}(\tau) d\tau+ c
\int_0^t(t-\tau)^{-1/4}A_{n,k+1}(\tau)d\tau.
\end{split}
\end{equation}
From \eqref{d81} and since $A_{n,k+1}\geq 0$, we obtain
\begin{equation}\label{d82}
\begin{split}
\int_0^t(t-\tau)^{-1/4}A_{n,k+1}(\tau)d\tau \leq&
(c+C_0)\int_0^t(t-\tau)^{-1/4}d\tau+cK_{f_n}\int_0^t(t-\tau)^{-1/4}\int_0^\tau
A_{n,k+1}(s) dsd\tau\\
&+ c
\int_0^t(t-\tau)^{-1/4}\int_0^\tau(\tau-s)^{-1/4}A_{n,k+1}(s)dsd\tau\\
\leq &c+cK_{f_n}\int_0^t(t-\tau)^{-1/4}\int_0^t
A_{n,k+1}(s) dsd\tau\\
&+ c \int_0^t(t-\tau)^{-1/4}\int_0^t(\tau-s)^{-1/4}A_{n,k+1}(s)dsd\tau\\
= &c+cK_{f_n}\int_0^t(t-\tau)^{-1/4}d\tau\int_0^t
A_{n,k+1}(s) ds\\
&+ c
\int_0^t\int_0^t(t-\tau)^{-1/4}(\tau-s)^{-1/4}A_{n,k+1}(s)dsd\tau\\
= &c+cK_{f_n}\int_0^t(t-\tau)^{-1/4}d\tau\int_0^t
A_{n,k+1}(s) ds\\
&+ c \int_0^t\Big [\int_0^t(t-\tau)^{-1/4}(\tau-s)^{-1/4}d\tau\Big
] A_{n,k+1}(s)ds\\
\leq &c+c\int_0^t A_{n,k+1}(s) ds,
\end{split}
\end{equation}
where we used that
$$\int_0^t(t-\tau)^{-1/4}d\tau<\infty,$$
and
$$\int_0^t(t-\tau)^{-1/4}(\tau-s)^{-1/4}d\tau
\leq
\Big[\int_0^t(t-\tau)^{-1/2}d\tau\Big]^{1/2}\Big[\int_0^t(\tau-s)^{-1/2}d\tau\Big{]}^{1/2}<\infty.$$
So, by using \eqref{d82} in \eqref{d81}, yields
\begin{equation}\label{d83}
\begin{split}
A_{n,k+1}(t) \leq& c+c\int_0^tA_{n,k+1}(\tau) d\tau,
\end{split}
\end{equation}
and by Gronwall's Lemma, we get
\begin{equation*}
\begin{split}
\sup_{i\leq k+1}\mathbf{E}\Big{(}\int_0^t\int_{\mathcal{D}}\|
D_{s,y}u_{n,i}(\cdot,t)\|_{L^\infty(\mathcal{D})}^2dyds\Big{)}=A_{n,k+1}(t)
\leq c=c(n),
\end{split}
\end{equation*}
which gives
\begin{equation}\label{d84}
\begin{split}
\sup_{t\in[0,T]}\sup_{i\leq
k+1}\mathbf{E}\Big{(}\int_0^T\int_{\mathcal{D}}\|
D_{s,y}u_{n,i}(\cdot,t)\|_{L^\infty(\mathcal{D})}^2dyds\Big{)}=\sup_{t\in[0,T]}\sup_{i\leq
k+1}\mathbf{E}\Big{(}\int_0^t\int_{\mathcal{D}}\|
D_{s,y}u_{n,i}(\cdot,t)\|_{L^\infty(\mathcal{D})}^2dyds\Big{)}<\infty.
\end{split}
\end{equation}
Here, we used that $D_{s,y}u_{n,i}(x,t)=0$ for any $s>t$ and thus
the integration is for $s\in[0,T]$, while we note that the bound
is independent of $k$. So, we have, by \eqref{d84}, that
\begin{equation*}
\begin{split}
\|u_{n,k+1}(x,t)\|_{D^{1,2}}:=&\Big{(}\mathbf{E}(|u_{n,k+1}(x,t)|^2)
+\mathbf{E}\Big (\int_0^T\int_\mathcal{D}|D_{s,y}u_{n,k+1}(x,t)|^2dyds\Big )\Big{)}^{1/2}\\
\leq& c+
c\Big{[}\sup_{t\in[0,T]}\mathbf{E}\Big{(}\int_0^T\int_{\mathcal{D}}\|
D_{s,y}u_{n,k+1}(\cdot,t)\|_{L^\infty(\mathcal{D})}^2dyds\Big{)}\Big{]}^{1/2}\\
\leq& c+c\Big{[}\sup_{t\in[0,T]}\sup_{i\leq
k+1}\mathbf{E}\Big{(}\int_0^T\int_{\mathcal{D}}\|
D_{s,y}u_{n,i}(\cdot,t)\|_{L^\infty(\mathcal{D})}^2dyds\Big{)}\Big{]}^{1/2}<\infty,
\end{split}
\end{equation*}
uniformly for any $k$; here, since $2<5$, we used the same
argument of proving \eqref{m12}, but for $2$ in place of $2pq$
(i.e., $\mathbf{E}(|u_{n,k+1}(x,t)|^2)<\infty$, the bound again
independent of $k$). This yields that
\begin{equation}\label{d85}
\exists\;D_{s,y}u_{n,k+1}(x,t)\;\in
D^{1,2}\;\;\forall\;(x,t)\in\mathcal{D}\times[0,T].
\end{equation}

Relations \eqref{d84}, \eqref{d85} complete the induction, and
establish \eqref{indmal}. \vspace{1.0cm}

As proved, for $p\geq 2$
$$\|u_n(\cdot,t)-u_{n,k}(\cdot,t)\|_{L^p(\Omega)}\rightarrow
0\;\;\mbox{as}\;\;k\rightarrow\infty,$$ and so,
\begin{equation}\label{nu1}
u_{n,k}(\cdot,t)\rightarrow u_n(\cdot,t)
\;\;\mbox{as}\;\;k\rightarrow\infty\;\;\mbox{in
the}\;L^2(\Omega)\;\mbox{norm},
\end{equation}
 while as we also proved
\begin{equation}\label{nu2}
u_{n,k}\in D^{1,2}\;\;\forall\;k.
\end{equation}
Moreover for
$$\|D_{\cdot,\cdot}u_{n,k}(x,t)\|_H:=\Big{[}\int_0^T\int_{\mathcal{D}}|D_{y,s}u_{n,k}(x,t)|^2dyds\Big{]}^{1/2},$$
it holds that
\begin{equation}\label{nu3}
\sup_{k}\mathbf{E}(\|D_{\cdot,\cdot}u_{n,k}(x,t)\|_H^2)<\infty,
\end{equation}
since by \eqref{indmal}
\begin{equation*}
\begin{split}
\sup_{k}\mathbf{E}(\|D_{\cdot,\cdot}u_{n,k}(x,t)\|_H^2)=&
\sup_{k}\mathbf{E}\Big(\int_0^T\int_{\mathcal{D}}|D_{s,y}u_{n,k}(x,t)|^2dyds\Big)\\
\leq&
\sup_{t\in[0,T]}\sup_{k}\mathbf{E}\Big{(}\int_0^T\int_{\mathcal{D}}\|
D_{s,y}u_{n,k}(\cdot,t)\|_{L^\infty(\mathcal{D})}^2dyds\Big{)}<\infty.
\end{split}
\end{equation*}
Using Lemma 1.2.3 of \cite{N}, due to \eqref{nu2}, \eqref{nu1},
\eqref{nu3}, we have the first result of this proposition, i.e.,
that
\begin{equation}\label{fres}
u_n(x,t)\in D^{1,2},
\end{equation}
and
$$D_{s,y}u_{n,k}(x,t)\rightarrow D_{s,y}u_{n}(x,t),$$
in the weak topology of $L^2(\Omega;H):=L^2(\Omega\times(
[0,T]\times\mathcal{D}))$, where
$$\|v\|_{L^2(\Omega;H)}:=\Big[\mathbf{E}\Big(\int_0^T\int_\mathcal{D}|v(y,s)|^2dyds\Big{)}\Big ]^{1/2}.$$
(Observe that for $(x,t)$ fixed, $D_{s,y}u_{n,k}(x,t)=v(y,s)$ for
some $v$.)

We remind that $D_{y,s}(u_{n,k+1}(x,t))$ was defined through
\eqref{mdk}. We shall show that $D_{s,y}u_{n}(x,t)$ satisfies
uniquely \eqref{MDE}.

Taking Malliavin derivatives in both sides of spde
\eqref{piecewise ff} (see the analogous calculus and arguments
for $D_{s,y}u_{n,k+1}$ given by \eqref{mdkn}), we obtain that for
any $s\leq t$
\begin{equation*}\label{mdkn2}
\begin{split}
D_{y,s}u_{n}(x,t)=&\int_{s}^{t}\int_{\mathcal{D}}[\Delta
G(x,z,t-\tau)-G(x,z,t-\tau)]
\tilde{\mathcal{G}_2}(n)(z,\tau)D_{y,s}u_{n,k}(z,\tau) dz d\tau\\
&+G(x,y,t-s)\sigma(u_{n}(y,s))\\
&+\int_{s}^{t}\int_{\mathcal{D}}G(x,z,t-\tau)\tilde{\mathcal{G}}_1(n)(z,\tau)D_{y,s}u_{n}(z,\tau)W(dz,d\tau),
\end{split}
\end{equation*}
i.e., \eqref{MDE} is satisfied, while for $s>t$
$$D_{y,s}u_{n}(x,t)=0.$$
Here, $\tilde{\mathcal{G}_1}(n)(z,\tau)$,
$\tilde{\mathcal{G}_2}(n)(z,\tau)$ are bounded, and satisfy
$$D_{y,s} (\sigma(u_n(z,\tau))) = \tilde{\mathcal{G}_1}(n)(z,\tau)D_{y,s}(u_n(z,\tau)),$$
$$D_{y,s} (f_n(u_n(z,\tau))) = \tilde{\mathcal{G}_2}(n)(z,\tau)D_{y,s}(u_n(z,\tau)).$$
Indeed, by Proposition 1.2.4 of \cite{N} (as we already used for
$u_{n,k+1}$), since $u_{n}$ belongs to $D^{1,2}$ and $\sigma$ is
Lipschitz uniformly on any $x$ in $\mathbb{R}$ with $K_\sigma$
its Lipschitz coefficient, then $\sigma(u_{n})$ belongs to
$D^{1,2}$ also, and there exists a random variable
$\tilde{\mathcal{G}}_1=\tilde{\mathcal{G}}_1(n)$ such that
\begin{equation}\label{form1n}
D_{y,s}\Big{(}\sigma(u_{n}(x,t))\Big{)}=\tilde{\mathcal{G}}_1(n)(x,t)D_{y,s}u_{n}(x,t),
\end{equation}
with $\tilde{\mathcal{G}}_1$ bounded (in the absolute value norm)
by $K_\sigma$, uniformly for any $x$, $t$, i.e.,
$$|\tilde{\mathcal{G}}_1(n)(x,t)|\leq
K_\sigma,\;\;\forall\;x\in\mathcal{D},\;\forall t\in[0,T].$$
Taking $f_n$ in place of $\sigma$, the same argument - since $f_n$
is also Lipschitz uniformly on $\mathbb{R}$ - yields
\begin{equation}\label{form2n}
D_{y,s}\Big{(}f_n(u_{n}(x,t))\Big{)}=\tilde{\mathcal{G}}_2(n)(x,t)D_{y,s}u_{n}(x,t),
\end{equation}
and
$$
|\tilde{\mathcal{G}}_2(n)(x,t)|\leq \hat{K}_{f_n},
$$
for $\hat{K}_{f_n}$ a positive constant, depending on $n$ through
$f_n$.

Remind that $\sigma$ is continuously differentiable and Lipschitz.

We note that as stated in the proof of Proposition 1.2.4 in
\cite{N}, since $f_n$ is continuously differentiable, then
$$\mathcal{G}_2(n,k)(x,t)=f_n'(u_{n,k}(x,t)),\;\;\;\tilde{\mathcal{G}}_2(n)(x,t)=f_n'(u_n(x,t)),$$
while for the same reason
$$\mathcal{G}_1(n,k)(x,t)=\sigma'(u_{n,k}(x,t)),\;\;\;\tilde{\mathcal{G}}_1(n)(x,t)=\sigma'(u_n(x,t)).$$

We need only to show uniqueness of solution of \eqref{MDE}; note
that from uniqueness of the Malliavin derivative,
$\tilde{\mathcal{G}_1}$, $\tilde{\mathcal{G}_2}$ are uniquely
determined. So, if $\hat{D}_{y,s}u_{n}(x,t)$ is another solution
of \eqref{MDE}, then through linearity of \eqref{MDE} on
$D_{y,s}u_{n}(x,t)$ or on $\hat{D}_{y,s}u_{n}(x,t)$, we get,
applying the same arguments, the analogous result as this for
\eqref{d83}. More specifically, for
$$B_{n}(t):=\mathbf{E}\Big{(}\int_0^t\int_{\mathcal{D}}\|
D_{s,y}u_{n}(\cdot,t)-\hat{D}_{s,y}u_{n}(\cdot,t)\|_{L^\infty(\mathcal{D})}^2dyds\Big{)},$$
we can analogously derive,
\begin{equation}\label{d83n}
\begin{split}
B_{n}(t) \leq& 0+c\int_0^tB_{n}(\tau) d\tau,
\end{split}
\end{equation}
and by Gronwall's Lemma we get that $B_n(t)=0$ for any $t$, i.e.,
$$\mathbf{E}\Big{(}\int_0^t\int_{\mathcal{D}}\|
D_{s,y}u_{n}(\cdot,t)-\hat{D}_{s,y}u_{n}(\cdot,t)\|_{L^\infty(\mathcal{D})}^2dyds\Big{)}=0,\;\;\forall\;t\in[0,T],
$$
which yields finally uniqueness of solution of \eqref{MDE}.

For $(x,t)$ given, we derive that
\begin{equation}\label{l12r}
\begin{split}
&\mathbf{E}\Big(\int_0^T\int_{\mathcal{D}}\int_0^T\int_\mathcal{D}|D_{s,y}u_n(x,t)|^2dydsdxdt\Big{)}
=\int_0^T\int_{\mathcal{D}}\mathbf{E}\Big(\int_0^T\int_\mathcal{D}|D_{s,y}u_n(x,t)|^2dyds\Big{)}dxdt\\
\leq&
c\int_0^T\int_{\mathcal{D}}\mathbf{E}\Big(\int_0^T\int_\mathcal{D}|D_{s,y}u_n(x,t)-D_{s,y}u_{n,k}(x,t)|
^2dyds\Big{)}dxdt\\
&+c\int_0^T\int_{\mathcal{D}}\mathbf{E}\Big(\int_0^T\int_\mathcal{D}|D_{s,y}u_{n,k}(x,t)|^2dyds\Big{)}dxdt\\
&\leq
c\int_0^T\int_{\mathcal{D}}\mathbf{E}\Big(\int_0^T\int_\mathcal{D}|D_{s,y}u_n(x,t)-D_{s,y}u_{n,k}(\cdot,t)|^2
dyds\Big{)}dxdt\\
&+c\int_0^T\int_{\mathcal{D}}\mathbf{E}\Big(\int_0^T\int_\mathcal{D}\|D_{s,y}u_{n,k}(\cdot,t)\|_
{L^\infty(\mathcal{D})}^2dyds\Big{)}dxdt<\infty,
\end{split}
\end{equation}
since, by \eqref{indmal}
$$\mathbf{E}\Big(\int_0^T\int_\mathcal{D}\|D_{s,y}u_{n,k}(\cdot,t)\|_
{L^\infty(\mathcal{D})}^2dyds\Big )<\infty,$$ and due to
\begin{equation}\label{nnn1}
\mathbf{E}\Big(\int_0^T\int_\mathcal{D}|D_{s,y}u_n(x,t)-D_{s,y}u_{n,k}(\cdot,t)|^2
dyds\Big{)}<\infty.
\end{equation}
In the previous argument we applied Fubini's Theorem. Moreover
\eqref{nnn1} holds true since $D_{s,y}u_{n,k}\rightarrow
D_{s,y}u_n$ as $k\rightarrow\infty$ in $L^2(\Omega)$; this
$L^2(\Omega)$ convergence can be easily established analogously
to the way that the $L^2(\Omega)$ convergence of $u_{n,k}$ was
established, i.e, we subtract the relation \eqref{mdkn} - which
defines the sequence of Malliavin derivatives $D_{s,y}u_{n,k}$,
and \eqref{MDE} - which is uniquely solvable for $D_{s,y}u_{n}$,
and derive after straight forward calculations, and since $f_n'$,
$\sigma'$ are continuous, the $L^2(\Omega)$ convergence of the
sequence of derivatives.

 Also, by the estimate \eqref{feq} of Lemma \ref{lemma 2.1}, we have
\begin{equation}\label{l13r}
\begin{split}
\mathbf{E}\Big(\int_0^T\int_\mathcal{D}|u_{n}(x,t)|^2dxdt\Big
)\leq& \int_0^T\mathbf{E}\Big(\int_\mathcal{D}|u_{n}(x,t)|^2dx\Big
)dt\\
\leq&\int_0^T\mathbf{E}\Big(\|u_{n}(\cdot,t)\|_{L^\infty(\mathcal{D})}^2\Big
)dt\\
\leq&\sup_{t\in[0,T]}\mathbf{E}\Big(\|u_{n}(\cdot,t)\|_{L^\infty(\mathcal{D})}^2\Big
)<\infty.
\end{split}
\end{equation}
Relations \eqref{l12r} and \eqref{l13r}, by definition, yield the
final regularity result of this proposition, i.e.,
\begin{equation}\label{eqeqww}
u_n(x,t)\in L^{1,2}.
\end{equation}
\end{proof}

The next Main Theorem is a direct consequence of the previous
arguments.
\begin{theorem}\label{loc}
Let $u$ be the solution of the stochastic Cahn-Hilliard/Allen-Cahn
equation \eqref{sm}, in dimension $d=1$, with smooth initial data
$u_0$. Moreover, let $\sigma$ satisfy for any $x\in\mathbb{R}$
\eqref{LEq}, i.e.,
\begin{equation*}
|\sigma(x)|\leq C(1+|x|^q),
\end{equation*}
for some $C >0$ and any $q \in (0,\frac{1}{3})$, and the
Lipschitz property on $\mathbb{R}$ \eqref{s3}, and also let
$\sigma$ be continuously differentiable on $\mathbb{R}$. Then the
solution $u$ of \eqref{sm} belongs to $L^{1,2}_{\rm loc}\subseteq
D^{1,2}_{\rm loc}$.
\end{theorem}
\begin{proof}
Indeed, since we constructed a localization of $u$, by
$(\Omega_n,u_n)$, $n\in \mathbb{N}$, with $u_n$ proven to be in
$L^{1,2}\subseteq D^{1,2}$.
\end{proof}

\begin{remark}\label{locd}
As already stated, the Malliavin derivative $D_{y,s}u$ is defined
well by the Malliavin derivatives of the restrictions
$u|_{\Omega_n}$ on $\Omega_n$:
$$D_{y,s}u:=D_{y,s}u_n,\;\;\mbox{on}\;\;\Omega_n.$$
\end{remark}
\section{Existence of a density for $u$}
In order to establish existence of a density for the solution $u$
of \eqref{sm}, we prove first the absolute continuity of the
approximation $u_n$.
\subsection{Absolute continuity of $u_n$}
Our aim is to prove that for $t>0$ and for $x\in[0,\pi]$
\begin{equation}\label{aim}
\|D_{\cdot,\cdot}u_n(x,t)\|_H^2=\int_0^t\int_\mathcal{D}|D_{y,s}u_n(x,t)|^2dyds>0,
\end{equation}
with probability $P=1$.

\begin{remark}\label{locderiv}
If we prove the above, then $\|D_{\cdot,\cdot}u_n(x,t)\|_H>0$
almost surely, while we have proved that $u_n\in D^{1,2}\subseteq
D_{\rm loc}^{1,2}\subseteq D_{\rm loc}^{1,1}$ (obviously applying
H\"older's inequality on the formula of
$\|\cdot\|_{D^{1,1}}$-norm where, cf. p. 27 of \cite{N},
$\|v||_{D^{1,1}}:=\mathbf{E}(|v|)+\mathbf{E}(\|D_{\cdot,\cdot}\|_H)$,
we see that $D^{1,2}\subseteq D^{1,1}$) and thus, according to
Theorem 2.1.3 of \cite{N} p. 98, $u_n$ is absolutely continuous
with respect to the Lebesgue measure on $\mathbb{R}$ (see also
the analogous argument used in \cite{CW1}).

Applying the same argument of Theorem 2.1.3 of \cite{N}, for $u$
this time, since by Theorem \ref{loc} $u\in D^{1,2}_{\rm loc}$,
in order to prove absolute continuity for $u$, we need to prove
that
$$\|D_{\cdot,\cdot}u\|_H>0, \;\;\;\;\mbox{almost surely}.$$

More specifically, the aforementioned Theorem 2.1.3 states: Let
$F$ be a random variable of the space $D_{\rm loc}^{1,1}$, and
suppose that $\|DF\|_H > 0$ a.s. Then the law of $F$ is absolutely
continuous with respect to the Lebesgue measure on $\mathbb{R}$.

In our case, we defined the space-time Malliavin derivative
operator $D:=D_{s,y}$ and $H:=L^2([0,T]\times\mathcal{D})$, while
we apply this theorem for $u_n$, $u$, for which we have shown that
$u_n\in D^{1,2}\subseteq D_{\rm loc}^{1,2}\subseteq D_{\rm
loc}^{1,1}$, $u\in D_{\rm loc}^{1,2}\subseteq D_{\rm loc}^{1,1}$.
\vspace{1.0 cm}

We prove now why the validity of \eqref{aim} is sufficient for
establishing the absolute continuity of $u$.

Let $\omega\in
A:=\cup_{k=1}^{n}\Omega_k=\Omega_n\subseteq\Omega$, then
$u(x,t;\omega)=u_n(x,t;\omega)$ a.s., and thus, for
\begin{equation*}
\begin{split}
&B:=\{\omega \in
\Omega:\;\|D_{\cdot,\cdot}u(x,t;\omega)\|_H>0\}\supseteq
C:=\{\omega \in A:\;\|D_{\cdot,\cdot}u(x,t;\omega)\|_H>0\},\\
&D_n:=\{\omega \in
A=\Omega_n:\;\|D_{\cdot,\cdot}u_n(x,t;\omega)\|_H>0\},
\end{split}
\end{equation*}
we have $P(C)=P(D_n)$ for any $n$. Set
$$Z:=\{\omega\in\Omega:\;\|D_{\cdot,\cdot}u_n(x,t;\omega)\|_H>0\}.$$
So, if \eqref{aim} is valid, then $P(Z)=1$, which gives
$P(Z^c)=0$.  But, observe that
$$D_n^c=\{\omega \in
A=\Omega_n:\;\|D_{\cdot,\cdot}u_n(x,t;\omega)\|_H\leq
0\}\subseteq \{\omega \in
\Omega:\;\|D_{\cdot,\cdot}u_n(x,t;\omega)\|_H\leq 0\}=Z^c,$$ so,
$$P(D_n^c)\leq P(Z^c)=0,$$
i.e., $P(D_n^c)=0$ and so $P(D_n)=1$. Thus, we have $$1\geq
P(B)\geq P(C)=P(D_n)=1,$$ which yields $P(B)=1$. Hence, indeed
$\|D_{\cdot,\cdot}u\|_H>0$, almost surely, and as already argued,
the law of $u$ is absolutely continuous with respect to the
Lebesgue measure on $\mathbb{R}$.
\end{remark}

In the sequel, we shall present two very important and difficult
estimates that are derived after treating carefully the growth of
the unbounded noise diffusion $\sigma$.
\begin{lemma}\label{leps}
Under the assumptions of Theorem \ref{loc}, the next estimates
hold true
\begin{equation}\label{d84un}
\sup_{t\in[\hat{s}-\varepsilon,\hat{s}]}\mathbf{E}\Big{(}\int_{\hat{s}-\varepsilon}^{\hat{s}}\int_{\mathcal{D}}\|
D_{s,y}u_{n}(\cdot,t)\|_{L^\infty(\mathcal{D})}^2dyds\Big{)}<C(n)\varepsilon^{2/3},
\end{equation}
and
\begin{equation}\label{d84un*}
\sup_{t\in[\varepsilon,T]}\mathbf{E}\Big{(}\int_{t-\varepsilon}^{t}\int_{\mathcal{D}}\|
D_{s,y}u_{n}(\cdot,t)-G(\cdot,y,t-s)\sigma(u_n(y,s))\|_{L^\infty(\mathcal{D})}^2dyds\Big{)}<C(n)\varepsilon^{17/12},
\end{equation}
 for any $\hat{s}\geq 0$, where
$\varepsilon<\min\{1,\hat{s}\}$, and $C(n)>0$ is a constant
independent of $t$, $\varepsilon$.
\end{lemma}
\begin{proof}
Using the spde \eqref{MDE} for the Malliavin derivative of $u_n$,
we proceed as when using equation \eqref{mdkn} (when we estimated
the Malliavin derivative of $u_{n,k}$) but integrating now on
$(a,t)$ for $t\geq a\geq 0$, instead of $(0,t)$. At the end, we
will use our result for $a:=0$ and for $a:=\hat{s}-\varepsilon$.

More specifically, for $p\geq 2$ , we get
\begin{equation*}
\begin{split}
 | D_{s,y}u_{n}(x,t)|^p \le &c|G(x,y,t-s)\sigma(u_{n}(y,s)) |^p \\
&+c\Big | \int_{s}^{t}\int_{\mathcal{D}}[\Delta G(x,z,
t-\tau)-G(x,z, t-\tau)] \tilde{\mathcal{G}}_2(n)(z,\tau)D_{s,y}u_{n}(z, \tau)  dz d \tau \Big |^p   \\
&+c \Big |\int_{s}^{t}\int_{\mathcal{D}}G(x,z, t-\tau)
\tilde{\mathcal{G}}_1(n)(z,\tau)D_{s,y} u_{n}(z,\tau) W(dz,d
\tau) \Big |^p,
\end{split}
\end{equation*}
which yields by the boundedness of $\tilde{\mathcal{G}}_2$
\begin{equation*}
\begin{split}
 \| D_{s,y}u_{n}(\cdot,t)\|_{L^\infty(\mathcal{D})}^p \le
 &c\|G(\cdot,y,t-s)\sigma(u_{n}(y,s)) \|_{L^\infty(\mathcal{D})}^p \\
&+c\Big \| \int_{s}^{t}\int_{\mathcal{D}}|\Delta G(\cdot,z,
t-\tau)-G(\cdot,z, t-\tau)|
|D_{s,y}u_{n}(z, \tau)|  dz d \tau \Big \|_{L^\infty(\mathcal{D})}^p   \\
&+c \Big \|\int_{s}^{t}\int_{\mathcal{D}}G(\cdot,z, t-\tau)
\tilde{\mathcal{G}}_1(n)(z,\tau)D_{s,y} u_{n}(z,\tau) W(dz,d
\tau) \Big \|_{L^\infty(\mathcal{D})}^p.
\end{split}
\end{equation*}
We integrate the previous for $y\in\mathcal{D},\;s\in [a,t]$ and
then take expectation, to derive
\begin{equation}\label{mmun}
\begin{split}
 \mathbf{E}\Big{(}&\int_{a}^t\int_{\mathcal{D}}
 \| D_{s,y}u_{n}(\cdot,t)\|_{L^\infty(\mathcal{D})}^pdyds\Big{)}
  \le c\mathbf{E}\Big{(}\int_{a}^t\int_{\mathcal{D}}\|G(\cdot,y,t-s)
  \sigma(u_{n}(y,s)) \|_{L^\infty(\mathcal{D})}^pdyds\Big{)} \\
&+c\mathbf{E}\Big{(}\int_{a}^t\int_{\mathcal{D}}\Big \|
\int_{s}^{t}\int_{\mathcal{D}}|\Delta G(\cdot,z,
t-\tau)-G(\cdot,z, t-\tau)| |D_{s,y}u_{n}(z, \tau)|  dz d \tau \Big \|_{L^\infty(\mathcal{D})}^pdyds\Big{)}   \\
&+c \mathbf{E}\Big{(}\int_{a}^t\int_{\mathcal{D}}\Big
\|\int_{s}^{t}\int_{\mathcal{D}}G(\cdot,z, t-\tau)
\tilde{\mathcal{G}}_1(n)(z,\tau)D_{s,y} u_{n}(z,\tau) W(dz,d
\tau)\Big \|_{L^\infty(\mathcal{D})}^pdyds\Big{)}\\
:=&E_1(t)+E_2(t)+E_3(t).
\end{split}
\end{equation}
We set $p=2$. We shall estimate the terms $E_{i}(t)$ for
$i=1,2,3$ when $p=2$.

We have for $1/\alpha+1/\beta=1$
\begin{equation}\label{d78un}
\begin{split}
E_{1}(t)=&c\mathbf{E}\Big{(}\int_{a}^t\int_{\mathcal{D}}\|G(\cdot,y,t-s)\sigma(u_{n}(y,s))
\|_{L^\infty(\mathcal{D})}^pdyds\Big{)}\\
\leq& c
\mathbf{E}\Big{(}\int_{a}^t\int_{\mathcal{D}}\|G(\cdot,y,t-s)\|_{L^\infty(\mathcal{D})}^{p\alpha}dyds\Big{)}
+c\mathbf{E}\Big{(}\int_{a}^t\int_{\mathcal{D}}|\sigma(u_{n}(y,s))
|^{p\beta}dyds\Big{)}\\
\leq& c
\mathbf{E}\Big{(}\int_{a}^t\int_{\mathcal{D}}\|G(\cdot,y,t-s)\|_{L^\infty(\mathcal{D})}^{p\alpha}dyds\Big{)}
+c\mathbf{E}\Big{(}\int_{a}^t\int_{\mathcal{D}}c(1+|u_{n}(y,s)|^{p\beta
q})
dyds\Big{)}\\
\leq& c
\mathbf{E}\Big{(}\int_{a}^t\int_{\mathcal{D}}\|G(\cdot,y,t-s)\|_{L^\infty(\mathcal{D})}^{p\alpha}dyds\Big{)}
+c(t-a)+c\mathbf{E}\Big{(}\int_{t-\varepsilon}^tc\|u_{n}(\cdot,s)\|_{L^\infty(\mathcal{D})}^{p\beta q}ds\Big{)}\\
\leq& c(t-a)+
\mathbf{E}\Big{(}\int_{a}^t\int_{\mathcal{D}}\|G(\cdot,y,t-s)\|_{L^\infty(\mathcal{D})}^{p\alpha}dyds\Big{)}
+c\int_{a}^t\mathbf{E}\Big{(}\|u_{n}(\cdot,s)\|_{L^\infty(\mathcal{D})}^{p\beta
q}\Big{)}ds,
\end{split}
\end{equation}
where we used the growth of the unbounded noise diffusion, for
$q\in(0,1/3)$, and Fubini's Theorem. We shall use $\alpha=7/6$,
$\beta=7$, and $p=2$.

By (1.6) of \cite{CW1}, we have for $d=1$, $p=2$
\begin{equation}\label{m11un}
\begin{split}
\int_{a}^{t}\int_{\mathcal{D}}\|G(\cdot,y,t-s)\|_{L^\infty(\mathcal{D})}^{p\alpha}dy
ds \le& C \int_{a}^t|t-s|^{-p\alpha d/4+d/4}ds=C
\int_{a}^t|t-s|^{-2(7/6)(1/4)+1/4}ds\\
=&C \int_{a}^t|t-s|^{-1/3}ds\leq C(t-a)^{2/3}.
\end{split}
\end{equation}

Also since $p\beta q=2\cdot 7\cdot q<2\cdot 7\cdot 1/3=14/3(<5)$,
in dimensions $d=1$, using \eqref{feq}, we obtain for any
$s\in[0,T]$ and thus for any $s\in[a,t]$
\begin{equation}\label{m12un}
\begin{split}
\mathbf{E}(\|u_{n}(\cdot,s)\|_{L^\infty(\mathcal{D})}^{p\beta
q})\leq c.
\end{split}
\end{equation}

Using \eqref{m11un}, \eqref{m12un} in \eqref{d78un}, yields
\begin{equation}\label{d79un1}
\begin{split}
E_{1}(t) \leq& c(t-a)+c
\mathbf{E}\Big{(}\int_{a}^t\int_{\mathcal{D}}\|G(\cdot,y,t-s)\|_{L^\infty(\mathcal{D})}^{p\alpha}dyds\Big{)}
+c\int_{a}^t\mathbf{E}\Big{(}\|u_{n}(\cdot,s)\|_{L^\infty(\mathcal{D})}^{p\beta q}\Big{)}ds\\
\leq& c(t-a)+c(t-a)^{2/3} +c(t-a)\leq c(t-a)+c(t-a)^{2/3},
\end{split}
\end{equation}
uniformly for all $t$, and thus for $p=2$
\begin{equation}\label{d79un}
E_{1}(t)\leq c(t-a)+c(t-a)^{2/3}.
\end{equation}

Considering the term $E_{2}(t)$, by (1.12) of \cite{CW1}, we have,
as in deriving \eqref{he2}, but observing that $s\leq \tau\leq t$
and $a\leq s\leq t$, which yields $s\in[a,\tau]$ when changing
the order of integration
\begin{equation}\label{he2un}
\begin{split}
E_{2}(t)=&c\mathbf{E}\Big{(}\int_{a}^t\int_{\mathcal{D}}\Big \|
\int_{s}^{t}\int_{\mathcal{D}}|\Delta G(\cdot,z,
t-\tau)-G(\cdot,z, t-\tau)||D_{s,y}u_{n}(z, \tau)|  dz d \tau
\Big \|_{L^\infty(\mathcal{D})}^pdyds\Big{)}\\
 \leq&c
\int_{a}^t\mathbf{E}\Big{(}\int_a^{\tau} \int_{\mathcal{D}}\|
D_{s,y}u_{n}(\cdot, \tau)\|^p_{L^\infty(\mathcal{D})} dyds\Big )
d\tau.
\end{split}
\end{equation}

For the term $E_3(t)$, Fubini's Theorem and
Burkholder-Davis-Gundy inequality, together with the boundedness
of $\tilde{\mathcal{G}}_1$, yields as in \eqref{he4}
\begin{equation}\label{he4un}
\begin{split}
E_3(t)=&c \mathbf{E}\Big{(}\int_{a}^t\int_{\mathcal{D}}\Big
\|\int_{s}^{t}\int_{\mathcal{D}}G(\cdot,z, t-\tau)
\tilde{\mathcal{G}}_1(n)(z,\tau)D_{s,y} u_{n}(z,\tau) W(dz,d
\tau)\Big \|_{L^\infty(\mathcal{D})}^pdyds\Big{)} \\
\leq &c \int_{a}^t\int_{\mathcal{D}}\mathbf{E}\Big{(}\Big
\|\int_{a}^{\tau}\int_{\mathcal{D}}|G(\cdot,z, t-\tau)|^2 |D_{s,y}
u_{n}(z,\tau)|^2 dzd \tau\Big
\|_{L^\infty(\mathcal{D})}^{p/2}\Big{)}dyds.
\end{split}
\end{equation}
As in \eqref{he5}, we derive since $d=1$ and $p=2$
\begin{equation}\label{he5un}
\begin{split}
E_3(t)\leq c \int_{a}^t(t-\tau)^{-1/4}\mathbf{E} \bigg
(\int_a^\tau\int_{\mathcal{D}}\|D_{s,y}
u_{n,k}(\cdot,\tau)\|_{L^\infty(\mathcal{D})}^2dyds\bigg )d\tau .
\end{split}
\end{equation}

Hence, by the estimates \eqref{d79un}, \eqref{he2un} and
\eqref{he5un}, we get for \eqref{mmun}
\begin{equation}\label{d80un}
\begin{split}
\mathbf{E}\Big{(}\int_{a}^t\int_{\mathcal{D}}\|
D_{s,y}u_{n}(\cdot,t)\|_{L^\infty(\mathcal{D})}^2dyds\Big{)} \leq&
c(t-a)+(t-a)^{2/3}\\
&+c\int_{a}^t\mathbf{E}\Big{(}\int_a^{\tau} \int_{\mathcal{D}}\|
D_{s,y}u_{n}(\cdot,
\tau)\|^2_{L^\infty(\mathcal{D})} dyds\Big ) d\tau\\
&+ c \int_{a}^t(t-\tau)^{-1/4}\mathbf{E} \bigg
(\int_a^\tau\int_{\mathcal{D}}\|D_{s,y}
u_{n}(\cdot,\tau)\|_{L^\infty(\mathcal{D})}^2dyds\bigg )d\tau,
\end{split}
\end{equation}
for $c>0$ constants independent of $t$.

Define
$$L_n(t):=\mathbf{E}\Big{(}\int_{a}^t\int_{\mathcal{D}}\|
D_{s,y}u_{n}(\cdot,t)\|_{L^\infty(\mathcal{D})}^2dyds\Big{)},$$
then \eqref{d80un} is written as
\begin{equation}\label{gr1}
\begin{split}
L_n(t)\leq c(t-a)+c(t-a)^{2/3}+c\int_{a}^tL_n(\tau) d\tau+ c
\int_{a}^t(t-\tau)^{-1/4}L_n(\tau)d\tau.
\end{split}
\end{equation}
This yields
\begin{equation*}
\begin{split}
c \int_{a}^t(t-\tau)^{-1/4}L_n(\tau)d\tau\leq&
c(t-a)^{1+1-1/4}+c(t-a)^{2/3+1-1/4}\\
&+c(t-a)^{1-1/4}\int_a^tL_n(s)ds\\
&+c\int_a^t[\int_a^t(t-\tau)^{-1/4}(\tau-s)^{-1/4}d\tau]L_n(s)ds\\
\leq&
c(t-a)^{1+1-1/4}+c(t-a)^{2/3+1-1/4}\\
&+c(t-a)^{1-1/4}\int_a^tL_n(s)ds+c(t-a)^{1/4}\int_a^tL_n(s)ds\\
\leq& c(t-a)^{7/4}+c(t-a)^{17/12}+c\int_a^tL_n(s)ds.
\end{split}
\end{equation*}
Thus \eqref{gr1} becomes
\begin{equation}\label{gr2}
\begin{split}
L_n(t)\leq C_0(t,a)+c\int_a^tL_n(\tau)d\tau,
\end{split}
\end{equation}
for
$$C_0(t,a):=c(t-a)+c(t-a)^{2/3}+c(t-a)^{7/4}+c(t-a)^{17/12}.$$
By \eqref{gr2}, we get
\begin{equation}\label{gr3}
\begin{split}
L_n(t)\leq C_0(t,a),
\end{split}
\end{equation}
 which yields,
\begin{equation}\label{d80un4}
\begin{split}
\mathbf{E}\Big{(}\int_{a}^t\int_{\mathcal{D}}\|
&D_{s,y}u_{n}(\cdot,t)\|_{L^\infty(\mathcal{D})}^2dyds\Big{)}=L_n(t)\\
&\leq c(t-a)+c(t-a)^{2/3}+c(t-a)^{7/4}+c(t-a)^{17/12},
\end{split}
\end{equation}
uniformly for any $t$.

In the above, $c=C(n)>0$ is independent of $t$, but generally may
depend on $n$. Since $D_{s,y}u_{n}(\cdot,t)=0$ when $s>t$, then
for any $\hat{s}\geq t$, by using \eqref{d80un4}, we have
\begin{equation}\label{d80unn}
\begin{split}
 \mathbf{E}\Big{(}\int_{a}^{\hat{s}}\int_{\mathcal{D}}\|
D_{s,y}u_{n}(\cdot,t)\|_{L^\infty(\mathcal{D})}^2dyds\Big{)}=&
\mathbf{E}\Big{(}\int_{a}^t\int_{\mathcal{D}}\|
D_{s,y}u_{n}(\cdot,t)\|_{L^\infty(\mathcal{D})}^2dyds\Big{)}\\
\leq& c(t-a)+c(t-a)^{2/3}+c(t-a)^{7/4}+c(t-a)^{17/12}\\
\leq&
c(\hat{s}-a)+c(\hat{s}-a)^{2/3}+c(\hat{s}-a)^{7/4}+c(\hat{s}-a)^{17/12}.
\end{split}
\end{equation}
So, choosing in the above $a:=\hat{s}-\varepsilon\leq t$ (we need
$a\leq t$), we have for any $\hat{s}\geq t\geq
\hat{s}-\varepsilon$,
\begin{equation}\label{d80unn}
\begin{split}
 \mathbf{E}\Big{(}\int_{\hat{s}-\varepsilon}^{\hat{s}}\int_{\mathcal{D}}\|
D_{s,y}u_{n}(\cdot,t)\|_{L^\infty(\mathcal{D})}^2dyds\Big{)} \leq
c\varepsilon^{2/3},
\end{split}
\end{equation}
for $\varepsilon<1$. Taking supremum on any such
$t\in[\hat{s}-\varepsilon,\hat{s}]$, we have the result, i.e.,
\eqref{d84un}.

Moreover, we have
\begin{equation}\label{d84un25}
\begin{split}
\mathbf{E}\Big{(}\int_{t-\varepsilon}^{t}\int_{\mathcal{D}}\|
D_{s,y}u_{n}(\cdot,t)-&G(\cdot,y,t-s)\sigma(u_n(y,s))\|_{L^\infty(\mathcal{D})}^2dyds\Big{)}\leq
cE_2(t)+cE_3(t)\\
&\leq
c\varepsilon^{2/3}\varepsilon+c\varepsilon^{2/3}\varepsilon^{1-1/4}\\
&\leq c\varepsilon^{17/12},
\end{split}
\end{equation}
where we used \eqref{he2un} and \eqref{he5un} for
$a=t-\varepsilon$ and the estimate \eqref{d84un}. So, the estimate
\eqref{d84un*} is established.

\end{proof}

We are now ready to prove the next important theorem, which will
yield by localization the second result of the Main Theorem
\ref{mres} of this paper. Here, we need a non-degenerating extra
assumption for the diffusion $\sigma$.
\begin{theorem}\label{denu0}
 Under the assumptions of Theorem \ref{loc}, if additionally,
$\sigma$ satisfies \eqref{s1}, i.e $$|\sigma(x)|\geq c_0>0,$$ for
any $x\in\mathbb{R}$, then the law of the solution $u_n(x,t)$ of
\eqref{piecewise ff} when $t>0$ and $x\in(0,\pi)$, is absolutely
continuous with respect to the Lebesgue measure on $\mathbb{R}$.
\end{theorem}
\begin{proof}
Relation \eqref{MDE} yields
\begin{equation} \label{MDEa}
\begin{split}
 |D_{y,s}u_{n}(x,t)|^2=&\Big{|}\int_{s}^{t}\int_{\mathcal{D}}[\Delta
G(x,z,t-\tau)-G(x,z,t-\tau)]
\tilde{\mathcal{G}_2}(n)(z,\tau)D_{y,s}(u_{n}(z,\tau)) dz d\tau\\
&+G(x,y,t-s)\sigma(u_{n}(y,s))\\
&+\int_{s}^{t}\int_{\mathcal{D}}G(x,z,t-\tau)\tilde{\mathcal{G}_1}(n)(z,\tau)D_{y,s}(u_{n}(z,\tau))W(dz,d\tau)\Big{|}^2\\
\geq& \frac{1}{2}A-B,
\end{split}
\end{equation}
for
$$A(x,y,s,t):=G(x,y,t-s)^2\sigma(u_{n}(y,s))^2,$$
and
\begin{equation*}
\begin{split}
B(x,y,s,t):=&\Big{|}\int_{s}^{t}\int_{\mathcal{D}}[\Delta
G(x,z,t-\tau)-G(x,z,t-\tau)]
\tilde{\mathcal{G}_2}(n)(z,\tau)D_{y,s}(u_{n}(z,\tau)) dz d\tau\\
&+\int_{s}^{t}\int_{\mathcal{D}}G(x,z,t-\tau)\tilde{\mathcal{G}_1}(n)
(z,\tau)D_{y,s}(u_{n}(z,\tau))W(dz,d\tau)\Big{|}^2,
\end{split}
\end{equation*}
where we used that $(2^{-1/2}a+2^{1/2}b)^2\geq 0$ which yields
$(a+b)^2\geq \frac{1}{2}a^2-b^2$.

So, we have
\begin{equation} \label{fin1}
\begin{split}
\int_0^t\int_{\mathcal{D}}
|D_{y,s}u_{n}(x,t)|^2dyds\geq\int_{t-\varepsilon}^t\int_{\mathcal{D}}
|D_{y,s}u_{n}(x,t)|^2dyds \geq
\frac{1}{2}\int_{t-\varepsilon}^t\int_{\mathcal{D}}Adyds-\int_{t-\varepsilon}^t\int_{\mathcal{D}}Bdyds.
\end{split}
\end{equation}

We will give an upper bound in expectation for the term
$$\int_{t-\varepsilon}^t\int_{\mathcal{D}}Bdyds.$$
For this, we shall use Lemma \ref{leps} (relation
\eqref{d84un*}), which yields
\begin{equation}\label{fin2}
\mathbf{E}\Big{(}\int_{t-\varepsilon}^t\int_{\mathcal{D}}B(x,y,s,t)dyds\Big{)}\leq
\mathbf{E}\Big{(}\int_{t-\varepsilon}^t\int_{\mathcal{D}}\|B(\cdot,y,s,t)\|_{L^\infty(\mathcal{D})}^2dyds\Big{)}
\leq C(n)\varepsilon^{17/12}.
\end{equation}

We now provide a lower bound for
$\int_{t-\varepsilon}^t\int_{\mathcal{D}}Adyds$. The
non-degeneracy condition of the diffusion and using the spectrum
in $[0,\pi]$ of the negative Neumann Laplacian, yields
\begin{equation}\label{fin3}
\begin{split}
\int_{t-\varepsilon}^t\int_{\mathcal{D}}Adyds=&
\int_{t-\varepsilon}^t\int_{\mathcal{D}}G(x,y,t-s)^2\sigma(u_{n}(y,s))^2dyds\\
\geq& c\int_{t-\varepsilon}^t\int_{\mathcal{D}}G(x,y,t-s)^2dyds=
c\int_{t-\varepsilon}^t\Big{[}\sum_{k=0}^\infty a_k^2(x)e^{-2(\lambda_k^2+\lambda_k)(t-s)}\Big{]}ds\\
\geq&c\int_{t-\varepsilon}^t\Big{[}\sum_{k=0}^\infty
a_k^2(x)e^{-4\lambda_k^2(t-s)}\Big{]}ds
=c\sum_{k=1}^\infty a_k^2(x)\frac{1}{4\lambda_k^2}[1-e^{-4\lambda_k^2\varepsilon}]+C\varepsilon\\
=&c\frac{1}{2}\sum_{k=1}^\infty
a_k^2(x)\frac{1}{2\lambda_k^2}[1-e^{-2\lambda_k^2(2\varepsilon)}]+C\varepsilon
\geq C(2\varepsilon)^{1-d/4}=C\varepsilon^{3/4},
\end{split}
\end{equation}
where we used the orthonormal $L^2(\mathcal{D})$ eigenfunctions
basis $\{a_k\}$ for $k=0,1,2,\cdots$, in dimensions $d=1$, the
fact that $\lambda_k=k^2$, and that $(t-s)\geq 0$ and that
$-(\lambda_k^2+\lambda_k)=-(k^4+k^2)\geq -2\lambda_k^2=-2k^4$, and
the estimate (3.25) of \cite{CW1}. Thus, we have proven that
\begin{equation}\label{fin4}
\begin{split}
\int_{t-\varepsilon}^t\int_{\mathcal{D}}Adyds\geq
C_0\varepsilon^{3/4}.
\end{split}
\end{equation}

Using  the estimates \eqref{fin2}, \eqref{fin4}, we arrive at
\begin{equation}\label{fin5}
\begin{split}
P\Big{(}\int_0^T\int_{\mathcal{D}}|D_{y,s}u_{n}(x,t)|^2dyds>0\Big{)}\geq&
P(\frac{1}{2}\int_{t-\varepsilon}^t\int_{\mathcal{D}}A(x,y,s,t)dyds-
\int_{t-\varepsilon}^t\int_{\mathcal{D}}B(x,y,s,t)dyds>0\Big{)}\\
\geq &P(
\int_{t-\varepsilon}^t\int_{\mathcal{D}}B(x,y,s,t)dyds< \frac{C_0}{2}\varepsilon^{3/4}\Big{)}\\
\geq&
1-c\mathbf{E}\Big{(}\int_{t-\varepsilon}^t\int_{\mathcal{D}}B(x,y,s,t)dyds\Big{)}\varepsilon^{-3/4}\\
\geq&1
-c\varepsilon^{17/12}\varepsilon^{-3/4}=1-c\varepsilon^{2/3}\rightarrow
1,\mbox{ as  }\varepsilon\rightarrow 0,
\end{split}
\end{equation}
where we applied Markov's inequality.

This yields the result.

We note that the proof of this theorem was influenced by the very
interesting arguments of Cardon-Weber in \cite{CW1}, for an
analogous result, where the stochastic Cahn-Hilliard equation with
bounded noise diffusion was considered. However, we used in a
direct way the property of $\sigma$, i.e., that $|\sigma(x)|\geq
c_0>0$ for any $x\in\mathbb{R}$.
\end{proof}
\subsection{Absolute continuity of the stochastic solution $u$}
The next theorem establishes the second result of Main Theorem
\ref{mres}, this of the existence of a density for $u$.
\begin{theorem}\label{denu}
Under the assumptions of Theorem \ref{loc}, if additionally,
$\sigma$ satisfies \eqref{s1}, i.e., $$|\sigma(x)|\geq c_0>0,$$ for
any $x\in\mathbb{R}$, then the law of the solution $u$ of
\eqref{sm} is absolutely continuous with respect to the Lebesgue
measure on $\mathbb{R}$.
\end{theorem}
\begin{proof}
This is a direct result of Theorem \ref{denu0} through
localization, see the arguments of Remark \ref{locderiv}.
\end{proof}

\end{document}